\documentclass[11pt,letterpaper]{amsart}
\usepackage{amssymb}
\usepackage{amsmath,amsthm}
\allowdisplaybreaks[1]

\theoremstyle{plain}
\newtheorem{thm}{Theorem} 
\newtheorem{lemma}{Lemma}
\newtheorem{cor}{Corollary}
\newtheorem{prop}{Proposition}
\newtheorem{conj}{Conjecture}
\theoremstyle{definition}
\newtheorem{remark}{Remark}
\newtheorem{example}{Example}
\usepackage{diagbox}
\usepackage{graphicx,color}

\newcommand{\QQ}{\mathbb{Q}}
\newcommand{\ZZ}{\mathbb{Z}}
\newcommand{\G}{\mathcal{G}}
\newcommand{\C}{\mathcal{C}}
\newcommand{\cP}{\mathcal{P}}
\newcommand{\Mgnbar}{\overline{\mathcal{M}}_{g,n}}
\newcommand{\Mgonebar}{\overline{\mathcal{M}}_{g,1}}
\newcommand{\p}{\partial}
\newcommand{\e}{\epsilon}
\newcommand{\beq}{\begin{equation}}
\newcommand{\eeq}{\end{equation}}
\newcommand{\nn}{\nonumber}

\begin{document}
\title[Large genus asymptotics of intersection numbers]{On the large genus asymptotics of \\  
 psi-class intersection numbers}
\author{Jindong Guo} 
\author{Di Yang}
\address{School of Mathematical Sciences, University of Science and Technology of China, 
230026, P.R.~China}
\email{guojindong@mail.ustc.edu.cn, diyang@ustc.edu.cn}

 
\maketitle

\begin{abstract}
Based on an explicit formula of the generating series for the $n$-point 
psi-class intersection numbers (cf.~Bertola et.~al.~\cite{BDY1}), 
we give a novel proof of a conjecture of Delecroix et.~al.~\cite{DGZZ19} 
regarding the large genus uniform leading asymptotics of the psi-class intersection numbers. 
We also investigate polynomiality phenomenon in the large genera.
\end{abstract}

\tableofcontents

\section{Introduction and statements of the results}\label{section1}
Let $g,n$ be non-negative integers satisfying the stability condition
\beq \label{stability}
2g-2+n>0,
\eeq and
$\overline{\mathcal{M}}_{g,n}$ the Deligne--Mumford moduli space~\cite{DM69} of stable algebraic curves 
of genus~$g$ with~$n$ distinct marked points. 
Denote by~$\mathcal{L}_j$ the $j$th cotangent line bundle on $\overline{\mathcal{M}}_{g,n}$, $j=1,\dots,n$, 
and $\psi_j:=c_1(\mathcal{L}_j)$ the first Chern class of~$\mathcal{L}_j$. 
The following integrals of products of psi-classes over $\overline{\mathcal{M}}_{g,n}$
\begin{equation}\label{intnumb}
\int_{\Mgnbar} \psi_1^{d_1} \cdots \psi_n^{d_n}
\end{equation}
are called {\it $n$-point psi-class intersection numbers of genus~$g$}. Here $d_1,\dots,d_n$ are nonnegative integers. 
According to the degree-dimension matching, the intersection numbers~\eqref{intnumb} vanish unless 
\beq\label{dd}
d_1+\cdots+d_n=3g-3+n.
\eeq

In 1990, Witten~\cite{Wi90} made a striking conjecture --- {\it the partition function~$Z=Z({\bf t};\e)$ of the psi-class intersection 
numbers~\eqref{intnumb}, defined by
\beq\label{WK}
Z({\bf t};\e)=\exp \Biggl(\sum_{g,n\geq0} \frac{\e^{2g-2}}{n!} \sum_{d_1,\dots,d_n\geq0} t_{d_1} \cdots t_{d_n} \int_{\Mgnbar} \psi_1^{d_1} \cdots \psi_n^{d_n} \Biggr), 
\eeq
is a particular tau-function for the Korteweg--de Vries (KdV) hierarchy with a special normalization of times. Here
${\bf t}=(t_0,t_1,t_2,\dots)$ is an infinite vector of indeterminates and $\e$ is an indeterminate.
In particular, the power series
\beq
u=u({\bf t};\e) := \e^2\frac{\partial^2\log Z({\bf t};\e)}{\partial t_0^2}
\eeq
satisfies the KdV equation:
\beq
\frac{\p u}{\p t_1} = \frac12 u \frac{\p u}{\p t_0} + \frac{\e^2}{12} \frac{\p^3 u}{\p t_0^3}.
\eeq Moreover, the partition function~$Z$ satisfies the following string and dilaton equations, respectively:
\begin{align}
& \sum_{d=0}^\infty t_{d+1} \frac{\p Z}{\p t_d} + \frac{t_0^2}{2\e^2} Z = \frac{\p Z}{\p t_0}, \label{stringeq}\\
& \sum_{d=0}^\infty t_d \frac{\p Z}{\p t_d} + \e \frac{\p Z}{\p \e} + \frac1{24}Z = \frac{\p Z}{\p t_1}.  \label{dilatoneq}
\end{align}}

Identities~\eqref{stringeq}--\eqref{dilatoneq} were proved by Witten~\cite{Wi90}.
It was shown by Dijkgraaf, Verlinde, Verlinde~\cite{DVV} (cf.~also~\cite{AvM}) 
that Witten's conjecture can be equivalently stated as follows: the partition function~$Z$ satisfies the following set of linear equations, called the {\it Virasoro constraints},
\begin{align}
L_m (Z) = 0 , \quad m\geq -1,
\end{align}
where $L_m$ are linear operators defined by
\begin{align}
&L_m := -\frac{(2m+3)!!}{2^{m+1}} \frac{\p}{\p t_{m+1}} + \sum_{d=0}^\infty \frac{(2d+2m+1)!!}{(2d-1)!! \, 2^{m+1}} 
t_d \frac{\p }{\p t_{d+m}} \label{vira}\\
& + \frac{\e^2} 2 \sum_{d=0}^{m-1} \frac{(2d+1)!!(2m-2d-1)!!}{2^{m+1}}\frac{\p^2}{\p t_d \p t_{m-1-d}} 
+ \frac{t_0^2}{2\e^2}\delta_{m,-1} + \frac1{16} \delta_{m,0}. \nn
\end{align}
The operators $L_m$ satisfy the Virasoro commutation relations~\cite{AvM,DVV}:
\beq
\left[L_{m_1},L_{m_2}\right] = (m_1-m_2) \, L_{m_1+m_2},\quad \forall\,m_1,m_2\geq -1. 
\eeq

Witten's conjecture was first proved by Kontsevich~\cite{K}, and is now called  
the {\it Witten--Kontsevich theorem}; see~\cite{AIS, KL, LX0, M, OP} 
for several different proofs of this theorem. 
The partition function~$Z({\bf t};\e)$ is now known as the {\it Witten--Kontsevich tau-function}, 
and we call 
\beq
\langle \tau_{d_1} \cdots \tau_{d_n} \rangle(\e) := \frac{\partial^n \log Z({\bf t};\e)}{\p t_{d_1} \dots \p t_{d_n}}\bigg|_{{\bf t}={\bf 0}}, 
\quad n,d_1,\dots,d_n\geq0,
\eeq
the $n$-point {\it Witten--Kontsevich correlators}. By definition, 
\beq\label{corrinter}
\langle \tau_{d_1} \cdots \tau_{d_n} \rangle(\e) = \sum_{g\geq0} \e^{2g-2} \int_{\Mgnbar} \psi_1^{d_1} \cdots \psi_n^{d_n}.
\eeq

Following Aggarwal~\cite{Agg1} and Delecroix--Goujard--Zograf--Zorich (DGZZ)~\cite{DGZZ19}, 
for $g,n,d_1,\dots,d_n$ being nonnegative integers satisfying $2g-2+n>0$ and $d_1+\cdots+d_n=3g-3+n$, 
 define the {\it normalized psi-class intersection numbers} $\mathcal{G}_{d_1,\dots,d_n}(g)$ as follows:
\begin{equation}\label{defiG}
\mathcal{G}_{d_1,\dots,d_n}(g):= \frac{24^g g! \prod_{j=1}^{n}(2d_j+1)!!}{(6g+2n-5)!!} \int_{\Mgnbar} \psi_1^{d_1} \cdots \psi_n^{d_n}.
\end{equation}
The following conjecture was made recently by Delecroix et.~al. in~\cite{DGZZ19}, 
which we will refer to as the {\it DGZZ conjecture}~\cite{DGZZ19} (cf.~also~\cite{Agg1,DGZZ20, DGZZ22}).

\smallskip

{\noindent {\bf Conjecture A.}} (DGZZ~\cite{DGZZ19})
{\it For an arbitrary positive number~$C<2$,
\begin{equation}\label{DGZZ0906}
\lim\limits_{g\to+\infty}\max_{n\in \mathbb{Z}^{\ge 1} \atop n \le C\log(g)} \max_{d_1,\dots,d_n\geq0 \atop d_1+\cdots+d_n=3g-3+n} 
\left|\mathcal{G}_{d_1,\dots,d_n}(g)-1\right|=0.
\end{equation}
}

We note that the leading asymptotics of the $n$-point psi-class intersection numbers 
was proved by K.~Liu and H.~Xu~\cite{LX} for the special case when $n,d_1,\dots,d_{n-1}$ are all fixed. 
We also note that the DGZZ conjecture is proved by Aggarwal~\cite{Agg1}. Actually, Aggarwal~\cite{Agg1} 
proves a stronger result that Conjecture~A still holds when ``$n<C\log(g)$" of~\eqref{DGZZ0906} 
is replaced by ``$n=o\bigl(g^{1/2}\bigr)$"; so due to Aggarwal's theorem, the number~$C$ 
in Conjecture~A can be an arbitrarily given positive number.
The first result of this paper is a novel proof of the following theorem, which we call the {\it DGZZ-A theorem}.
\begin{thm}[DGZZ-A]\label{thm1}
For arbitrary $C>0$, formula~\eqref{DGZZ0906} is true .
\end{thm}
\noindent The proof is given in Section~\ref{section2}. 

Here we provide some ideas of the proof. 
Recall that the {\it matrix-resolvent method of computing logarithmic derivatives of tau-functions} for  
integrable systems is developed 
in~\cite{BDY1,BDY2,BDY3} (see also~\cite{DVY, DY1, DY2, DY3, DYZ0, DYZ, Zhou}).
(It was pointed out in~\cite{BDY1, DY2} with explicit examples 
that for topological tau-functions this method is efficient when the genus is large.) 
In particular, using the Witten--Kontsevich theorem and the matrix-resolvent method for the KdV hierarchy, 
M.~Bertola, B.~Dubrovin and the second author of the present paper 
derived in~\cite{BDY1} an explicit formula of certain generating series for the 
$n$-point Witten--Kontsevich correlators or equivalently for the $n$-point psi-class intersection numbers 
(cf.~\eqref{corrinter} and~\eqref{dd}); see~also \cite{BDY2,DYZ,Zhou}. 
More precisely, for $n\geq1$, denote by $C_n(\lambda_1,\dots,\lambda_n)$ the following formal 
power series of $\lambda_1^{-1},\dots,\lambda_n^{-1}$:
\begin{equation}\label{defCn0921}
C_n(\lambda_1,\dots, \lambda_n)=\sum_{g,d_1,\dots,d_n\ge0}\frac{\prod_{j=1}^{n}(2d_j+1)!!}{\lambda_1^{d_1+1}\cdots\lambda_n^{d_n+1}}
\int_{\Mgnbar} \psi_1^{d_1} \cdots \psi_n^{d_n}.
\end{equation}
Then the formula from~\cite{BDY1} reads as follows:
\begin{align} 
& C_1(\lambda)=\sum_{g\geq1}\frac{(6g-3)!!}{24^g \, g! \, \lambda^{3g-1}}, \label{f1mr}\\
& C_n(\lambda_1,\dots,\lambda_n)= - \frac1{n} \sum_{\sigma\in S_n}\frac{{\rm tr}\big(M(\lambda_{\sigma(1)})\cdots M(\lambda_{\sigma(n)})\big)}{\prod_{i=1}^{n}(\lambda_{\sigma(i)}-\lambda_{\sigma(i+1)})} \label{f2mr} \\
& \qquad\qquad\qquad\qquad - \delta_{n,2}\frac{\lambda_1+\lambda_2}{(\lambda_1-\lambda_2)^2} \quad (n\ge 2), \nn
\end{align}
where $S_n$ denotes the symmetric group, for an element $\sigma\in S_n$, $\sigma(n+1)$ is understood as $\sigma(1)$, and 
$M(\lambda)$ denotes the following particular element of $\mathfrak{sl}(2,\QQ((\lambda^{-1})))$
\begin{equation}\label{defM1002}
M(\lambda) := 
\frac{1}{2}
\begin{pmatrix}
-\sum_{g=1}^{\infty}\frac{(6g-5)!!}{24^{g-1}(g-1)!}\lambda^{-3g+2} & -2\sum_{g=0}^{\infty}\frac{(6g-1)!!}{24^{g}g!}\lambda^{-3g} \\ 2\sum_{g=0}^{\infty}\frac{6g+1}{6g-1}\frac{(6g-1)!!}{24^{g}g!}\lambda^{-3g+1} & \sum_{g=1}^{\infty}\frac{(6g-5)!!}{24^{g-1}(g-1)!}\lambda^{-3g+2}
\end{pmatrix}.
\end{equation}
Here, we note that $(-1)!!=1$.
For $n=1$, formula~\eqref{f1mr}, or equivalently, 
\beq
\int_{\Mgonebar} \psi_1^{3g-2} = \frac{1}{24^g \, g!}, \quad g\geq1
\eeq
is well known. For $n\geq2$, 
to understand the 
right-hand side of~\eqref{f2mr} as power series of $\lambda_1^{-1},\dots,\lambda_n^{-1}$,
we introduce as in~\cite{BDY1} the notation
\beq
P(\sigma;\lambda_1,\dots,\lambda_n)=\prod_{q=1}^n \frac1{\lambda_{\sigma(q)}-\lambda_{\sigma(q+1)}}, \quad \sigma\in S_n,
\eeq
and we need to   
 perform the Laurent expansions around~$\infty$s (large $|\lambda_1|$, \dots, $|\lambda_n|$) of the rational function 
 $P(\sigma;\lambda_1,\dots,\lambda_n)$, $\sigma\in S_n$,  
and of $(\lambda_1+\lambda_2)/(\lambda_1-\lambda_2)^2$ (when $n=2$), 
within a certain region having a fixed ordering between $|\lambda_1|,\dots,|\lambda_n|$ satisfying  
$|\lambda_i|\neq |\lambda_j|$ for all $i\neq j$. It is shown in~\cite{BDY1,BDY3,DYZ} that the resulting formal Laurent series of the whole right-hand side of~\eqref{f2mr} is independent of 
the ordering (see~\cite{DYZ0} for a direct proof). 
Below for simplicity we fix the choice of the region to be $|\lambda_1|>\dots>|\lambda_n|$. We have the following lemma. 
\begin{lemma} \label{LaurentP}
For each $\sigma\in S_n$, the Laurent expansion of the rational function $P(\sigma;\lambda_1,\dots,\lambda_n)$ around~$\infty$s within the region  
$|\lambda_1|>\dots>|\lambda_n|$ is given by
\beq
P(\sigma;\lambda_1,\dots,\lambda_n) = (-1)^{m(\sigma)} \sum_{j_1,\dots,j_n\geq0} \prod_{q=1}^n 
\lambda_{\sigma(q)}^{J_{\sigma,q}(j_q)-J_{\sigma,q-1}(j_{q-1})-1}.
\eeq
Here, 
\beq\label{msdef}
m(\sigma):= {\rm card} \left\{q \in \{1,\dots,n\} \mid \sigma(q+1)<\sigma(q) \right\},
\eeq
and 
\begin{equation}\label{defjsi}
J_{\sigma,q}(j) := 
\left\{\begin{array}{ll}
-j-1,  &\sigma(q)<\sigma(q+1), \\
j,   &\sigma(q)>\sigma(q+1),
\end{array}\right.
\end{equation}
where $\sigma(0)$ is understood as $\sigma(n)$.
\end{lemma}
\noindent 
Explicit formulae for the $n$-point psi-class intersection numbers will then be derived
 in Proposition~\ref{prop1} of Section~\ref{section2}
using~\eqref{f2mr} and Lemma~\ref{LaurentP} 
as we described above, whose estimates  
will lead to Theorem~\ref{thm1}.

\begin{remark}
Both Aggarwal's proof~\cite{Agg1} and our proof use the Witten--Kontsevich theorem. However, unlike Aggarwal's proof where the {\it Virasoro constraints}~\eqref{vira} are used, 
our proof uses~\eqref{f1mr}--\eqref{f2mr}.  Also, in Aggarwal's proof a more sophisticated technique using probability theory is applied, while, as we shall see the estimates in our proof are more straightforward. 
\end{remark}

\begin{remark}
There are explicit formulae~\cite{Bur,K,O}
for several other types of generating series for the same intersection numbers~\eqref{intnumb}, 
with their relationships being discussed in~\cite{BDY1} (cf.~also~\cite{Zhou}); 
nevertheless, it seems to us that the explicit formula given by~\eqref{f1mr}--\eqref{f2mr} for generating series of the type~\eqref{defCn0921}
is the more suitable one towards the large genus asymptotics (cf.~\cite{BDY1,CMZ,DY1,DY2}).
\end{remark}

The behavior for $\G_{d_1,\dots,d_n}(g)$ in large~$g$ is easier to understand 
when $n$, $d_1,\dots,d_{n-1}$ are all {\it fixed} ($d_n=3g-3+n-d_1-\cdots-d_{n-1}$). 
For this situation, using \eqref{f1mr}--\eqref{f2mr} we will give in Section~\ref{section3} a new proof of  
 the following theorem, which can also be deduced easily from  
 the work of Liu and~Xu~\cite{LX}. So we refer to this theorem as {\it Liu--Xu's theorem}. 
\begin{thm} [Liu--Xu~\cite{LX}] \label{thm2}
For fixed $n\geq1$ and fixed $d_1,\dots,d_{n-1}\geq0$ being integers, write $|d|=d_1+\cdots+d_{n-1}$.
Then we have the asymptotic expansion
\begin{equation} \label{37}
\G_{d_1,\dots,d_{n-1},3g-3+n-|d|}(g) \sim 
\sum_{k=0}^{\infty}\frac{\mathcal{G}_{k}}{g^k} \quad (g\to \infty),
\end{equation}
where $\mathcal{G}_0=1$, and $\mathcal{G}_{k}=\G_k(n,d_1,\dots,d_{n-1})$, $k\geq 1$, are constants. 
Moreover, there exists a rational function $R(g;d_1,\dots,d_{n-1})$ of~$g$
whose coefficients may depend on $d_1,\dots,d_{n-1}$, 
such that 
\beq
\G_{d_1,\dots ,d_{n-1},3g-3+n-|d|}(g)=R(g;d_1,\dots,d_{n-1})
\eeq 
(so the above asymptotic expansion is convergent). 
Furthermore, the rational function $R(g;d_1,\dots,d_{n-1})$ has at most~$|d|$ possible poles at 
$-(2n-5)/6$, $-(2n-7)/6$, \dots, $-(2n-3-2|d|)/6$. 
\end{thm}
\noindent We note that a weaker form of Theorem~\ref{thm2} (giving $\G_0$ and $\G_1$ in the asymptotic expansion~\eqref{37}) also  
follows easily from the results of~\cite{DGZZ20}.

Let us now proceed to discuss a nice and new property for psi-class 
intersection numbers, which is described  
in the following two conjectures (Conjecture~\ref{conj1011} and Conjecture~\ref{conj1}). We will call it 
the {\it polynomiality phenomenon in the large genera}. This phenomenon as well as   
the integrality phenomenon found in~\cite{DYZ2} give the {\it arithmetic}  
perspective of intersection numbers in the large genera.

\begin{conj} \label{conj1011}
There exist a sequence of polynomials 
\begin{equation}
G_k\Bigl(n,p_0,\dots,p_{\left[\frac32 k\right]-1}\Bigr)\in\mathbb{Q}\left[n,p_0,\dots,p_{\left[\frac32 k\right]-1}\right], \quad k\geq 0
\end{equation} 
with $G_0(n)=1$, such that, 
for arbitrary fixed $n\geq 1$ and 
fixed $d_1,\dots,d_{n-1}$ being non-negative integers, 
 the following equalities hold: 
\beq\label{GkequalsGk}
\G_k(n,d_1,\dots,d_{n-1}) = G_k\Bigl(n,p_0,\dots,p_{\left[\frac32 k\right]-1}\Bigr),
\eeq
where $\G_k(n,d_1,\dots,d_{n-1})$ are the constants  
introduced in the above Theorem~\ref{thm2}, and
$p_i$ denotes the multiplicity of~$i$ in $(d_1,\dots,d_{n-1})$. 
Moreover, under the following degree assignments 
\begin{equation}
 \deg \, n=1,\quad \deg \, p_i=i+1,\quad i\geq 0,
\end{equation}
the polynomials $G_k(n,p_0,\dots,p_{[3k/2]-1})$ satisfy the degree estimates
\begin{equation}
\deg \, G_k\Bigl(n,p_0,\dots,p_{\left[\frac32 k\right]-1}\Bigr)\leq 2k
\end{equation} 
and contain the following explicit terms
\begin{equation}
 \frac{1}{12^kk!}\left(n^{2k}+(-1)^{k}p_0^{2k}\right).
 \end{equation}
\end{conj}
If Conjecture~\ref{conj1011} holds, then it follows from 
the rationality-in-$g$ statement of Theorem~\ref{thm2}   
that the polynomials $G_k(n,p_0,\dots,p_{[3k/2]-1})$, $k\geq0$,
contain {\it all} the information of the normalized intersection numbers $\G_{d_1,\dots,d_n}(g)$. 

\begin{example} \label{example1}
Assuming the validity of Conjecture~\ref{conj1011}, 
one can use Theorem~\ref{thm2}, \eqref{GkequalsGk}, and \eqref{f1mr}, \eqref{f2mr}
to determine the polynomials $G_k$ ($k\geq1$). We list in below the first several of them: 
{\small
\begin{align}
&G_1(n,p_0)=\frac{(n-1)(n-6)+(5-p_0)p_0}{12}, \label{g1expression}\\
&G_2(n,p_0,p_1,p_2)= \label{g2expression}\\
&\frac{(n-1)(3n^3-59n^2+298n-228)}{864}+\frac{ p_0(346-390n+30n^2)}{864} \nn \\ 
&+\frac{ p_0^2(69+78n-6n^2)-46p_0^3+3p_0^4-p_1(204-180p_0+36p_0^2)-60p_2}{864}, \nn\\
&G_3(n,p_0,p_1,p_2,p_3) = \label{g3expression} \\
&\frac{n^6-41n^5+555n^4-3031n^3+6092n^2-5160n+1584}{10368} \nn \\
&+\frac{-p_0^6+31p_0^5+3p_0^4(n^2-19n-73+12p_1)}{10368} \nn \\
&-\frac{p_0^3(46n^2-874n+552p_1+120p_2+127)}{10368} \nn \\
&+\frac{p_0^2 \bigl(-3n^4+98n^3-36n^2(p_1+20)+n(684p_1-1253)\bigr)}{10368} \nn\\
&+\frac{p_0^2(-54p_1^2+312p_1+285p_2+409)}{2592} \nn \\
&+\frac{p_0 \bigl(15n^4-490n^3+n^2(4291+180p_1)-12n(572+285p_1)\bigr)}{10368} \nn\\
&+\frac{p_0(90p_1^2+171p_1-285p_2-70p_3+258)}{864} \nn \\
&-\frac{102p_1^2+p_1(17n^2-323n+60p_2+402)+5(n^2p_2-19np_2-28p_3)}{864}.\nn
\end{align}}

\noindent We note that a concrete algorithm in~\cite{DY1,DY2} and the string and dilaton 
equations \eqref{stringeq}, \eqref{dilatoneq} (cf.~\eqref{stringapp}, \eqref{1011114})
could help to facilitate the computations. 
\end{example}

If $d_1,\dots,d_{n-1}$ are not fixed, we have a stronger conjectural statement.
\begin{conj}\label{conj1}
For each fixed $n\geq 1$ and for arbitrary $K\ge1$, 
\begin{align} 
&\lim\limits_{g\to+\infty} \max_{d_1,\dots,d_n\geq0 \atop d_1+\cdots+d_n=3g-3+n} 
g^K\left|\mathcal{G}_{d_1,\dots,d_n}(g)-\sum_{k=0}^{K}
\frac{G_k\Bigl(n,p_0,\dots,p_{\left[\frac32 k\right]-1}\Bigr)}{g^k}\right|=0, \label{seriesG} 
\end{align}
where $G_k$, $k\geq0$, denote the polynomials in 
Conjecture~\ref{conj1011}, 
and $p_i$ denotes the multiplicity of~$i$ in $(d_1,\dots,d_n)$.
\end{conj}

\begin{remark}
Recall that according to Aggarwal~\cite{Agg1}, 
the DGZZ-A theorem (i.e.~Theorem~\ref{thm1}) can lead to a beautiful proof  
of another DGZZ conjecture~\cite{DGZZ19} regarding 
the large genus asymptotics of Masur--Veech volumes of the moduli space of quadratic differentials. 
It would then be interesting if the validity of Conjecture~\ref{conj1} 
%or~\eqref{suspectlargeg} 
could imply the refinement given in~\cite{YZZ}.
\end{remark}

The next theorem partially supports the validity of Conjecture~\ref{conj1}.
\begin{thm} \label{thm3}
There exist a sequence of absolute functions 
$$P_k(n), \quad k\geq 0$$
with $P_0(n)\equiv1$, such that, for arbitrary $K\geq1$ and 
for every fixed $n\geq1$,
\begin{equation}
\lim_{g\to\infty} \max_{\substack{d_1,\dots,d_n\ge \left[\frac{3K}2\right]\\ d_1+\cdots+d_n=3g+n-3}} g^K 
\left|\G_{d_1,\dots ,d_n}(g)-\sum_{k=0}^K \frac{P_k(n)}{g^k}\right|=0.
\label{thm3formula}
\end{equation}
\end{thm}

\noindent The proof of this theorem is in Section~\ref{section4}.

\smallskip

The paper is organized as follows. In Section~\ref{section2} we prove Theorem~\ref{thm1}. 
In Section~\ref{section3} we prove Theorem~\ref{thm2}. 
In Section~\ref{section4} we prove Theorem~\ref{thm3}.
Proofs of several lemmas are given in Appendix~\ref{app0918}. 

\medskip

\noindent {\bf Acknowledgements} We are grateful to Don Zagier for several very helpful suggestions. 
We also thank the anonymous referee for very helpful comments that improve the presentation of the paper.
The work is partially supported by 
the National Key R and D Program of China 2020YFA0713100 (``Analysis and Geometry on Bundles") 
and NSFC 12061131014.

\section{Psi-class intersection numbers and proof of Theorem~\ref{thm1}}\label{section2}
In this section we prove Theorem~\ref{thm1}.
We first do a series of preparations.

\smallskip

\noindent {\it Notations.}

1. For $k_1,\dots,k_n\geq-1$, we use $a_{k_1,\dots,k_n}$ to denote 
certain rational numbers, defined by 
\begin{equation} \label{a}
a_{k_1,\dots,k_n}:=\left\{\begin{array}{ll}
2b_{k_1}\cdots b_{k_n},  &\quad P_1,\\
& \\
(-1)^{\sum_{j=0}^{s}(i_{2j+1}-i_{2j}-1)}b_{k_1}\cdots b_{k_n},  &\quad P_2, \\
& \\
(-1)^{\sum_{j=1}^{s}(i_{2j}-i_{2j-1}-1)}b_{k_1}\cdots b_{k_n},  &\quad P_3,\\
& \\
0, &\quad {\rm otherwise}.
\end{array}
\right.
\end{equation}
Here, $b_k$, $k\geq-1$, are given by
	\begin{equation}\label{defbk}
		b_k:=\left\{
		\begin{array}{ll}
			-\frac{(6g-1)!!}{24^gg!},  &k=3g, \\
			\\
			\frac{6g+1}{6g-1}\frac{(6g-1)!!}{24^gg!},  &k=3g-1, \\
			\\
			\frac{1}{2}\frac{(6g-5)!!}{24^{g-1}(g-1)!},   &k=3g-2, \\
		\end{array}
		\right.
	\end{equation}
and the conditions $P_1,P_2,P_3$ in the case distinction are given by
\begin{itemize}
\item[$P_1$:] $n$ is even and $k_i\equiv 1\,({\rm mod}\,3)$ for all $i=1,\dots,n$,
\item[$P_2$:] there exist $i_1<\dots<i_{2s}$, such that $k_{i_{2j-1}}\equiv0\,({\rm mod}\,3)$, $k_{i_{2j}}\equiv 2\,({\rm mod}\,3)$, $\forall\, j=1,\dots,s$, 
and $k_{t}\equiv 1\,({\rm mod}\,3)$, $\forall \, t\notin \{{i_1}, \dots ,i_{2s}\}$, setting $i_0=0$, $i_{2s+1}=n+1$,
\item[$P_3$:] there exist $i_1<\dots< i_{2s}$, such that $k_{i_{2j-1}}\equiv2\,({\rm mod}\,3)$, $k_{i_{2j}}\equiv 0\,({\rm mod}\,3)$, $\forall\, j=1,\dots,s$, 
and $k_{t}\equiv 1\,({\rm mod}\,3)$, $\forall\,t\notin \{{i_1},\dots ,i_{2s}\}$.
\end{itemize}
We also make the {\it convention} that $a_{k_1,\dots,k_n}:=0$ if one of $k_1,\dots,k_n$ is smaller than or equal to~$-2$. 
It is clear from~\eqref{a} that $a_{k_1,\dots,k_n}$ is invariant with respect to the cyclic permutations 
of its indices, and it vanishes unless
\beq\label{kisumgn}
k_1+\cdots+k_n=3g-3+n
\eeq
for some $g\in \ZZ$.

\smallskip

2. For $n\geq2$, $\underline{d}=(d_1,\dots,d_n)\in \left(\ZZ^{\ge 0}\right)^n$, $\sigma\in S_n$, $q=1,\dots,n$, 
define $K_{\underline{d}, \sigma, q}:  \left(\ZZ^{\ge 0}\right)^n\rightarrow \ZZ$ by 
\begin{align} 
& K_{\underline{d}, \sigma, q} (\underline{j}):=d_{\sigma(q)}+J_{\sigma,q}(j_q)-J_{\sigma,q-1}(j_{q-1}), \label{defK927}
\end{align}
where $\underline{j}=(j_1,\dots,j_n) \in \left(\ZZ^{\ge 0}\right)^n$, and $J_{\sigma,q}$ are given in~\eqref{defjsi}.

\smallskip

3. For $g\geq0$, $n\geq2$, and for $k_1,\dots,k_{n}\geq-1$ being integers satisfying~\eqref{kisumgn},
denote 
\begin{equation}\label{defikappa}
\kappa_{k_1,\dots,k_n}(g):=\frac{24^{g+\left[\frac{n}2\right]-1} \left(g+\left[\frac{n}2\right]-1\right)!}{ \left(6g+2 \left[\frac{3n}2\right]-7\right)!!}
a_{k_1,\dots,k_n}.
\end{equation}

\smallskip

4. For $n\ge2$, $\underline{d}\in \left(\ZZ^{\ge 0}\right)^n$, $\sigma\in S_n$, and $r_1,r_2\in\ZZ$, denote 
\begin{align}
& V^{r_2}_{\underline{d},\sigma,r_1}
:= \left\{\underline{j}\in \left(\ZZ^{\geq0}\right)^n \Big|
r_1\leq \max_{1\leq q\leq n} K_{\underline{d}, \sigma,q}(\underline{j}) \leq r_2\right\} \label{Vr1r2def}\\
&\quad\quad\quad\quad\quad\bigcap\left\{\underline{j}\in \left(\ZZ^{\geq0}\right)^n\Big| a_{K_{\underline{d}, \sigma,1}(\underline{j}),\dots,K_{\underline{d}, \sigma,n}(\underline{j})}\neq0 \right\}.\nn 
\end{align}

\smallskip

5. For $n\ge 2$, $1\le r\le n-1$ being integers, 
we call a permutation $\sigma\in S_n$ with $\sigma(1)=1$ 
an {\it $(r,n-r)$-permutation}, if
\begin{align}
	& 1=\sigma(1)<\sigma(2)<\dots <\sigma(r+1)=n,\\
	& n>\sigma(r+2)>\sigma(r+3)>\cdots >\sigma(n)>1
\end{align}
(namely, {\it firstly increasing then decreasing}, often called {\it unimodal} permutations).
The set of all permutations of this type is denoted by $S_{n,1}^{(r,n-r)}$. We also denote 
$S_{n,1}^{\{2\}}=\cup_{r=1}^{n-1}S_{n,1}^{(r,n-r)}$. 
For $n\geq2$, $l\geq1$, denote by 
$S_{n,1}^{\{2l\}}$ the set of all permutations $\sigma \in S_n$ with $\sigma(1)=1$ satisfying that 
there exist $1=t_1<\dots<t_{2l}\leq n$ such that
\begin{align}
& \sigma(t_{2u-1})<\sigma(t_{2u-1}+1)<\cdots<\sigma(t_{2u}),\label{defSn2l1} \\
& \sigma(t_{2u})>\sigma(t_{2u}+1)>\cdots>\sigma(t_{2u+1}) \label{defSn2l2}
\end{align}
for all $1\leq u\leq l$. Here we set $t_{2l+1}=t_1$. 
Let us list a few elementary combinatorial facts that will be used later:
\begin{align}
& {\rm card} \, S_{n,1}^{(r,n-r)}=\binom{n-2}{r-1},\\
& {\rm card} \, S_{n,1}^{\{2\}}=2^{n-2}, \label{numberofeleinsn1}\\
& {\rm card} \, S_{n,1}^{\{2l\}} \leq  (2l)^{n-1}. \label{estimatesn12l}
\end{align}

\smallskip

Using the method described in Section~\ref{section1} and the above notations we will prove the following proposition. 

\begin{prop}\label{prop1}
For $n\ge 2$, $g\geq0$, and $\underline{d}=(d_1,\dots,d_n)\in \left(\ZZ^{\ge 0}\right)^n$ satisfying $d_1+\dots+d_n=3g+n-3$, the following formula is true:
\begin{equation} \label{formulaWK}
\G_{\underline{d}}(g) = \frac1n \sum_{\sigma\in S_n} \gamma_{\underline{d},\sigma}(g),
\end{equation}
where
\begin{align} \label{defepsilon}  
& \gamma_{\underline{d},\sigma}(g) 
:= \frac{24^g g!\, (-1)^{m(\sigma)+1}}{(6g-5+2n)!!} \sum_{\underline{j}\in (\mathbb{Z}^{\ge0})^n}
a_{K_{\underline{d},\sigma,1}(\underline{j}),\dots,K_{\underline{d},\sigma,n}(\underline{j})}.
\end{align}
\end{prop}

\noindent Here, we note that there are only a finite number of nonzero summands on the 
right-hand side of~\eqref{defepsilon}. Indeed, by the above given convention it suffices to indicate that
there are finite many $\underline{j}\in \left(\ZZ^{\ge0}\right)^n$ such that $K_{\underline{d},\sigma,q}(\underline{j})\ge-1$ for all $q=1,\dots,n$. Due to the cyclic invariance of $\gamma_{\underline{d},\sigma}(g)$, we can assume $\sigma(1)=1$, say,
 $\sigma\in S_{n,1}^{\{2l\}}$. Let $t_1,\dots,t_{2l}$ be defined by \eqref{defSn2l1}, \eqref{defSn2l2}. If $s\in\ZZ$ satisfies $t_{2u-1}\leq s<t_{2u}$ for some $u$, then from~\eqref{defK927} and~\eqref{defjsi} we have 
\begin{align}
-n\leq \sum_{q=t_{2u-1}}^{s}K_{\underline{d},\sigma,q}(\underline{j})
%&=J_{\sigma,s}(j_s)-J_{\sigma,t_{2u-1}-1}(j_{t_{2u-1}-1})+\sum_{q=t_{2u-1}}^{s}d_{\sigma(q)}\nn\\
&=-j_{t_{2u-1}-1}-j_{s}-1+\sum_{q=t_{2u-1}}^{s}d_{\sigma(q)} \leq 3g+n-4-j_s,\nn
\end{align}
which gives $j_s\leq 3g+2n-4$. If $s\in\ZZ$ satisfies $t_{2u}\leq s<t_{2u+1}$ for some~$u$, we have $-n\leq \sum_{q=s+1}^{t_{2u+1}}K_{\underline{d},\sigma,q}(\underline{j})\leq 3g+n-4-j_s$, which gives $j_s\leq 3g+2n-4$. 
The needed finiteness follows.

\begin{remark}
Formula~\eqref{formulaWK} with $n=2$ is given in~\cite{BDY1}, which is shown in~\cite{Guo}  
to be equivalent to Zograf's formula~\cite{Zograf} 
for the 2-point intersection numbers. 
\end{remark}

\begin{proof}[Proof of Proposition~\ref{prop1}]
The matrix $M(\lambda)$ defined in~\eqref{defM1002} can be written as
\begin{equation}
M(\lambda)=\sum_{k=-1}^{\infty}A_k\lambda^{-k},
\end{equation}
where $A_k$, $k\geq-1$, are matrices given by 
\begin{equation}
A_k=\left\{
\begin{array}{ll}
			\begin{pmatrix}
				0 & -\frac{(6g-1)!!}{24^gg!} \\ 0 & 0
			\end{pmatrix}, \qquad &k=3g, \\
			\\
			\begin{pmatrix}
				0 & 0 \\ \frac{6g+1}{6g-1}\frac{(6g-1)!!}{24^gg!} & 0
			\end{pmatrix}, \qquad &k=3g-1, \\
			\\
			\begin{pmatrix}
				-\frac{1}{2}\frac{(6g-5)!!}{24^{g-1}(g-1)!} & 0 \\ 0 & \frac{1}{2}\frac{(6g-5)!!}{24^{g-1}(g-1)!}
			\end{pmatrix}, \qquad &k=3g-2. \\
\end{array}\right.
\end{equation}
Then we have the following identity: 
\begin{equation}\label{defia}
a_{k_1,\dots,k_n}=\mathrm{tr} \, (A_{k_1}\cdots A_{k_n}), \quad k_1,\dots,k_n\ge -1.
\end{equation}
The details of the proof of this identity are given in Appendix~\ref{app0918}.
The proposition can be proved by using~\eqref{f2mr}, \eqref{defia} and Lemma~\ref{LaurentP}.
\end{proof}
We note that the cyclic symmetry of $a_{k_1,\dots,k_n}$ is more obvious from~\eqref{defia}.

The following five lemmas give certain estimates about $\kappa_{k_1,\dots,k_n}(g)$ defined in~\eqref{defikappa}. 
 Their proofs 
 %of these lemmas 
 are put in Appendix~\ref{app0918}. We will use $\ZZ_{\rm even}^{\ge k}$ ($\ZZ_{\rm odd}^{\ge k}$) to denote the sets of even (odd) integers that are bigger than or equal to~$k$.

\begin{lemma} \label{lem1even0906}
For every $m\geq0$, there exists a constant $C_1=C_1(m)$, such that for $g$ being sufficiently large,
\beq
\label{lem1evenestimate}
\sup_{n\in \mathbb{Z}_{\rm even}^{\geq2}}
\max_{\substack{ -1 \le k_1,\dots,k_n\le 3g+\frac{3n}2-3-m \\ k_1+\cdots+k_n=3g-3+n}}  g^{\left[\frac23 m\right]}
\left|\kappa_{k_1,\dots,k_n}(g)\right| \leq C_1,
\eeq
\end{lemma}

\begin{lemma} \label{lem1odd0906}
For every $m\geq0$, there exists a constant $C_5=C_5(m)$, such that for $g$ being sufficiently large, 
\beq
\sup_{n\in\mathbb{Z}_{\rm odd}^{\geq3}} 
\max_{\substack{ -1\le k_1,\dots,k_n \le 3g+\frac{3n-1}2-3-m \\ k_1+\cdots+k_n=3g-3+n}} g^{\left[\frac{2m+1}3\right]} \left|\kappa_{k_1,\dots,k_n}(g)\right| \leq C_5.
\eeq
\end{lemma}

\begin{lemma} \label{lem2} 
For every positive real number~$C$, 
there exists a constant $C_6=C_6(C)$, such that for $g$ being sufficiently large,
\begin{equation}
\max_{n\in\mathbb{Z}^{\geq 2} \atop n\leq C\log(g)} \max_{1\leq m\leq n}\max\limits_{\substack{k_1,\dots,k_n\geq -1 \\ k_1+\cdots+k_n=3g+n-3 \\ 
		{\rm card} \{ i \mid k_i\geq 1\}\ge m}} \left|\kappa_{k_1,\dots,k_n}(g)\right| \frac{g^{m-3}}
{14^n} \prod_{j=1}^n (k_j+2)^2 \leq C_6.
\end{equation} 
\end{lemma}

\begin{lemma}\label{lem3}
For every positive real number~$C$ and 
every $m\in\ZZ^{\ge1}$, there exists a constant $C_8=C_8(C,m)$, such that for $g$ being sufficiently large,
\begin{equation}\label{ineqlemma5}
\max_{n\in\mathbb{Z}_{\rm even}^{\ge2} \atop n\leq C\log(g)} 
\max\limits_{\substack{k_1,\dots,k_n\ge-1\\k_1,\dots,k_n\leq 3g+\frac{3n}2-3-m\\ k_1+\cdots+k_n=3g+n-3}} 
\left|\kappa_{k_1,\dots,k_n}(g)\right|\frac{g^{\left[\frac23 m\right]-2}}{2^n}\prod_{j=1}^n (k_j+2)^2 \leq C_8.
\end{equation} 
\end{lemma}

\begin{lemma}\label{lem3'}
For every positive real number~$C$ and every $m\in\ZZ^{\ge 0}$, there exists a constant $C_{11}=C_{11}(C, m)$,  
such that for $g$ being sufficiently large,
\begin{equation}
\max_{n\in\mathbb{Z}_{\rm odd}^{\ge3} \atop n\le C\log(g)} \! 
\max\limits_{\substack{k_1,\dots,k_n\ge-1\\k_1,\dots,k_n\leq 3g+\frac{3n-1}2-3-m\\ k_1+\cdots+k_n=3g+n-3}} \!\!\!\!\!
 \tfrac{\left|\kappa_{k_1,\dots,k_n}(g)\right| \prod_{j=1}^n (k_j+2)^2}{2^n}  \, g^{\left[\frac{2m+1}3\right]-2} \leq C_{11}.
\end{equation} 
\end{lemma}

We note that Lemma~\ref{lem1even0906} and Lemma~\ref{lem1odd0906} provide estimates for $\left|\kappa_{k_1,\dots,k_n}(g)\right|$ for $n$ even and $n$ odd respectively. Lemma~\ref{lem3} is similar to Lemma~\ref{lem1even0906}, and Lemma~\ref{lem3'} is similar to Lemma~\ref{lem1odd0906}.  These lemmas will be 
used in  
the proofs of Lemmas \ref{lem6}, \ref{lem7} (see below). 
Let us now prove the following important lemma.
\begin{lemma}\label{lem6}
For every positive real number $C$,
\begin{equation}\label{lemma7ineq}
\lim\limits_{g\to+\infty}\max_{2\leq n\leq C\log(g)}\max_{d_1,\dots,d_n\ge 0 \atop d_1+\cdots+d_n=3g+n-3 }\left|\sum_{\sigma\in S_{n,1}^{\{2\}}}\gamma_{\underline{d},\sigma}(g)-1\right|=0. 
\end{equation}
\end{lemma}

\begin{proof}
Consider the case when $n$ is {\it even}. Take $D=8+[2C]$. By using~\eqref{defepsilon} and the triangle inequality we have
\begin{align}\label{sum}
&\left|\sum_{\sigma\in S_{n,1}^{\{2\}}}\gamma_{\underline{d},\sigma}(g)-1\right|\\
& \leq \left|\frac{24^g g!}{(6g+2n-5)!!}\sum_{\sigma\in S_{n,1}^{\{2\}}}(-1)^{m(\sigma)+1}\sum_{\underline{j} \in V^{3g+\frac{3n}2-D}_{\underline{d},\sigma,-1}}a_{K_{\underline{d},\sigma,1}(\underline{j}),\dots,K_{\underline{d},\sigma,n}(\underline{j})}\right|\nn\\ 
&+\left|\frac{24^gg!}{(6g+2n-5)!!}\sum_{\sigma\in S_{n,1}^{\{2\}}}(-1)^{m(\sigma)+1}\sum_{\underline{j} \in V^{3g+\frac{3n}2-6}_{\underline{d},\sigma,3g+\frac{3n}2-D+1}}a_{K_{\underline{d},\sigma,1}(\underline{j}),\dots,K_{\underline{d},\sigma,n}(\underline{j})}\right|\nn\\
&+\left|\frac{24^gg!}{(6g+2n-5)!!}\sum_{\sigma\in S_{n,1}^{\{2\}}}(-1)^{m(\sigma)+1}\sum_{\underline{j} \in V^{3g+\frac{3n}2-3}_{\underline{d},\sigma,3g+\frac{3n}2-5}}a_{K_{\underline{d},\sigma,1}(\underline{j}),\dots,K_{\underline{d},\sigma,n}(\underline{j})}-1\right|.\nn
\end{align}
In the above formulae we applied $\underline{j}\in V_{\underline{d},\sigma,-1}^{3g+\frac{3n}{2}-3}$, due to the following facts:
\begin{align}%\label{sumK}
K_{\underline{d},\sigma,1}(\underline{j})+\cdots+K_{\underline{d},\sigma,n}(\underline{j})=d_1+\cdots+d_n=3g+n-3\nn
\end{align} and $a_{K_{\underline{d},\sigma,1}(\underline{j}),\dots,{K_{\underline{d},\sigma,n}(\underline{j})}}$ vanishes when $K_{\underline{d},\sigma,q}(\underline{j})=K_{\underline{d},\sigma,q+1}(\underline{j})=-1$ for some $1\leq q\leq n-1$ or $K_{\underline{d},\sigma,n}(\underline{j})=K_{\underline{d},\sigma,1}(\underline{j})=-1$.
Let us start with estimating the first term on the right-hand side of~\eqref{sum}. 
For each $\sigma\in S_{n,1}^{\{2\}}$, say, $\sigma\in S_{n,1}^{(r,n-r)}$, we have
\begin{align}
& \left|\sum_{\underline{j}\in V_{\underline{d},\sigma,-1}^{3g+\frac{3n}2-D}} a_{K_{\underline{d},\sigma,1}(\underline{j}),\dots,K_{\underline{d},\sigma,n}(\underline{j})} \right| \label{lemma8formula1} \\
& \leq \frac{(6g+3n-7)!!}{24^{g+\frac{n}2-1}\left(g+\frac{n}2-1\right)!} \sum_{\underline{j}\in V_{\underline{d},\sigma,-1}^{3g+\frac{3n}2-D}} \left|\kappa_{K_{\underline{d},\sigma,1}(\underline{j}),\dots,K_{\underline{d},\sigma,n}(\underline{j})}\right|\nn\\
& \leq \frac{C_8 \,(6g+3n-7)!!}{24^{g+\frac{n}2-1}\left(g+\frac{n}2-1\right)!} \, \frac{2^n}{g^{\left[\frac23(D-3)\right]-2}}\sum_{\underline{j}\in V_{\underline{d},\sigma,-1}^{3g+\frac{3n}2-D}} \frac{1}{\prod_{q=1}^n(K_{\underline{d},\sigma,q}(\underline{j})+2)^2} \nn\\
& \leq \frac{C_8 \,(6g+3n-7)!!}{24^{g+\frac{n}2-1}\left(g+\frac{n}2-1\right)!} \, \frac{2^n}{g^{\left[\frac23(D-3)\right]-2}} \nn\\
& \qquad \times \sum_{\underline{j}\in V_{\underline{d},\sigma,-1}^{3g+\frac{3n}2-D}} \frac{(K_{\underline{d},\sigma,r+1}(\underline{j})+2)^2}{\prod_{q=1}^n(K_{\underline{d},\sigma,q}(\underline{j})+2)^2} \frac{1}{(j_r+3)^2} \nn\\
& \leq \frac{C_8 \,(6g+3n-7)!!}{24^{g+\frac{n}2-1}\left(g+\frac{n}2-1\right)!} \, \frac{2^n}{g^{\left[\frac23 (D-3)\right]-2}}\left(\frac{\pi^2}6\right)^n. \nn
\end{align}
Here, in the first inequality we substituted~\eqref{defikappa},
in the second inequality we used the $m=D-3$ case of~\eqref{ineqlemma5} of Lemma~\ref{lem3} where 
we recall that $C_8=C_8(C,D-3)$, in the third inequality we used the fact that 
$$K_{\underline{d},\sigma,r+1}(\underline{j})+2=d_n+j_r+j_{r+1}+3\geq j_r+3,$$
and in the fourth inequality we used~\eqref{defK927}.
Therefore, by using~\eqref{numberofeleinsn1} we have 
\begin{align}
& \left|\sum_{\sigma\in S_{n,1}^{\{2\}}}(-1)^{m(\sigma)+1} \sum_{\underline{j}\in V_{\underline{d},\sigma,-1}^{3g+\frac{3n}2-D}} a_{K_{\underline{d},\sigma,1}(\underline{j}),\dots,K_{\underline{d},\sigma,n}(\underline{j})} \right| \label{lemma8formula11010} \\
& \leq \frac{C_8(C,D-3) \,(6g+3n-7)!!}{24^{g+\frac{n}2-1}\left(g+\frac{n}2-1\right)!} \, \frac{2^{2n-2}}{g^{\left[\frac23 (D-3)\right]-2}}\left(\frac{\pi^2}6\right)^n.\nn
\end{align}

Before proceeding, we mention that the following two combinatorial 
statements would be helpful:
\begin{itemize}
\item[(A)]  Given $k_1,\dots,k_n\geq -1$ satisfying $k_1+\cdots+k_n=3g+n-3$.
Equations 
\[
K_{\underline{d},\sigma,q}(\underline{j})=k_q,\quad q=1,\dots,n
\]
for $\underline{j} \in\left(\ZZ^{\ge0}\right)^n$
have solutions only if $k_{r+1}\geq 1$, and have at most $k_{r+1}$ solutions. 
Here we remind the reader that $\sigma\in S_{n,1}^{(r,n-r)}$. 
\item[(B)] Given $m\ge0$, $k_1,\dots,k_n\geq -1$ 
satisfying $k_1+\cdots+k_n=3g+n-3$ and 
$$\sum_{1\le q\le n\atop k_q\ge1} k_q=m+k_{r+1}.$$
Equations
\[
K_{\underline{d},\sigma,q}(\underline{j})=k_q,\quad q=1,\dots,n
\]
for $\underline{j}\in\left(\ZZ^{\ge0}\right)^n$
have at least $d_{1}-2m-k_1$ solutions,
and have at most $d_{1}-k_1$ solutions.
\end{itemize} 

We now estimate the second term on the right-hand side of~\eqref{sum}. We have
\begin{align}
	\sum_{\underline{j} \in V^{3g+\frac{3n}2-6}_{\underline{d},\sigma,3g+\frac{3n}2-D+1}}  \!\!\!\!\!\!\!\!\!
	a_{K_{\underline{d},\sigma,1(\underline{j})},\dots,K_{\underline{d},\sigma,1(\underline{j})}} 
	= \sum_{f=6}^{D-1} \! \sum_{\underline{j} \in V^{3g+\frac{3n}2-f}_{\underline{d},\sigma,3g+\frac{3n}2-f}}  \!\!\!\!\!\!
	a_{K_{\underline{d},\sigma,1(\underline{j})},\dots,K_{\underline{d},\sigma,1(\underline{j})}} .
\end{align}
For each $f=6,\dots,D-1$, taking $m=f-3$ in Lemma~\ref{lem1even0906} we know that  
\begin{align}
	\max_{\substack{-1\le k_1,\dots,k_n\le 3g+\frac{3n}2-f \\ k_1+\cdots+k_n=3g-3+n}}  g^{\left[\frac{2}{3}(f-3)\right]}
	\left|\kappa_{k_1,\dots,k_n}(g)\right| < C_{1},
\end{align}
where we recall that $C_1=C_{1}(f-3)$ is independent of~$n$. 
With the help of the statements (A) and (B), one can deduce that the number of 
elements~$\underline{j}$ in $V^{3g+3n/2-f}_{\underline{d},\sigma,3g+3n/2-f}$ 
can be controlled by a function of the form
\beq
d_1 A(f,n)+B(f,n),
\eeq
where $A(f,n)$ and $B(f,n)$, for each~$f$, are certain polynomial functions of~$n$. Therefore, 
\begin{align}
& \left|\sum_{\underline{j} \in V^{3g+\frac{3n}2-6}_{\underline{d},\sigma,3g+\frac{3n}2-D+1}} 
a_{K_{\underline{d},\sigma,1(\underline{j})},\dots,K_{\underline{d},\sigma,1(\underline{j})}} \right| \label{lem6estimate5}\\
& \leq \frac{(6g+3n-7)!!}{24^{g+\frac{n}2-1}\left(g+\frac{n}2-1\right)!} \sum_{f=6}^{D-1} \left| d_1 A(f,n)+B(f,n) \right| \frac{C_1(f-3)}{g^{\left[\frac23(f-3)\right]}}.\nn
\end{align}

To estimate the third term
on the right-hand side of~\eqref{sum},  we will divide the 
consideration into two cases: the $n\ge3$ case and the $n=2$ case.  For $n\ge3$,
we decompose $V_{\underline{d},\sigma,3g+3n/2-5}^{3g+3n/2-3}$ as follows:
$$
V_{\underline{d},\sigma,3g+\frac{3n}2-5}^{3g+\frac{3n}2-3} = V_{\underline{d},\sigma,3g+\frac{3n}2-4}^{3g+\frac{3n}2-3} \bigsqcup 
V_{\underline{d},\sigma,3g+\frac{3n}2-5}^{3g+\frac{3n}2-5}. 
$$
Notice that all possible $(k_1,\dots,k_n)$ satisfying $a_{k_1,\dots,k_n}\neq0$ 
and $\max_{1\leq j\leq n}k_j\ge 3g+\frac{3n}{2}-5$ are up to cyclic permutations given by the following four classes:
\begin{align}
& \Bigl(3g+\frac{3n}2-3,-1,0,\dots,-1,0,-1\Bigr), & \Bigl(3g+\frac{3n}2-4,0,-1,\dots,0,-1,0\Bigr), \nn\\
& \Bigl(3g+\frac{3n}2-5,-1,0,\dots,1,\dots,-1,0\Bigr), & \Bigl(3g+\frac{3n}2-5,0,-1,\dots,1,\dots,0,-1\Bigr). \nn
\end{align} 
Here, in the third class, it starts with $3g+\frac{3n}2-5$, followed by $-1,0,\dots,-1,0$, with~1 being inserted in any of the positions after $3g+\frac{3n}2-5$; the fourth class is similar, with $-1,0$ replaced by $0,-1$.
For each of such $(k_1,\dots,k_n)$, the number of solutions to $K_{\underline{d},\sigma,q}(\underline{j})=k_q$, $q=1,\dots,n$ can be estimated by the statements (A) and~(B).
Then by using \eqref{a}, \eqref{Vr1r2def} we obtain that
\begin{align}
&  \sum_{\sigma\in S_{n,1}^{\{2\}}}(-1)^{m(\sigma)+1}\sum_{\underline{j}\in V_{\underline{d},\sigma,3g+\frac{3n}2-4}^{3g+\frac{3n}2-3}} 
a_{K_{\underline{d},\sigma,1}(\underline{j}),\dots,K_{\underline{d},\sigma,n}(\underline{j})},\label{lem6estimate2}\\ 
& =\frac{2^{n-2}\,(6g+3n-7)!!}{24^{g+\frac{n}2-1}(g+\frac{n}2-1)!}\left(1+\frac{1}{6g+3n-7}\right), \nn\\
&\left|\sum_{\sigma\in S_{n,1}^{\{2\}}}(-1)^{m(\sigma)+1}\sum_{\underline{j}\in V_{\underline{d},\sigma,3g+\frac{3n}2-5}^{3g+\frac{3n}2-5}}
a_{K_{\underline{d},\sigma,1}(\underline{j}),\dots,K_{\underline{d},\sigma,n}(\underline{j})}\right|\label{lem6estimate3}\\
& \leq (5n-5) \, \frac{2^{n-2} \, (6g+3n-11)!!}{4^1 \, 24^{g+\frac{n}2-2}\left(g+\frac{n}2-2\right)!}. \nonumber
\end{align}
For $n=2$, we have
\begin{align}
&\left|(-1)^{m({\rm id})+1}\sum_{\underline{j}\in V_{(d_1,d_2),{\rm id},3g-2}^{3g}} 
	a_{d_1-1-j_1-j_2, d_2+1+j_1+j_2}-\frac{(6g-1)!!}{24^gg!}\right|\label{lem6estimate3n=2}\\
	&\leq \max\left\{d_1\frac{6(6g-5)!!}{24^gg!},\left|d_1\frac{6(6g-5)!!}{24^gg!}-\frac{36g(6g-5)!!}{24^gg!}\right|\right\}. \nn
\end{align}

We conclude from the above 
\eqref{sum}, \eqref{lemma8formula11010}, \eqref{lem6estimate5}, 
\eqref{lem6estimate2}, \eqref{lem6estimate3}, \eqref{lem6estimate3n=2} 
that 
\begin{equation}
\lim\limits_{g\to+\infty}\max_{n\in \ZZ^{\geq 2}_{\rm even} \atop n\leq C\log(g)}\max_{d_1,\dots,d_n\ge 0 \atop d_1+\cdots+d_n=3g+n-3 }\left|\sum_{\sigma\in S_{n,1}^{\{2\}}}\gamma_{\underline{d},\sigma}(g)-1\right|=0, 
\end{equation}
where for the case $n\geq3$ we also used the following elementary facts:
\begin{align}
& \frac{(6g+3n-7)!!}{24^{g+\frac{n}2-1}\left(g+\frac{n}2-1\right)!}=\frac{(6g+2n-5)!!}{2^{n-2}24^gg!}\prod_{j=1}^{\frac{n}2-1}\frac{(6g+2n-5+2j)}{6(g+j)}, \label{21} \\
& \forall \, \epsilon>0, \quad \lim_{g\to\infty}\max_{n\in \mathbb{Z}^{\ge3}_{\rm even} \atop 
n\leq C\log(g)}g^{1-\epsilon}\left|
\prod_{j=1}^{\frac{n}2-1}\frac{6g+2n-5+2j}{6(g+j)}-1\right|=0. \label{lem6estimate6}
\end{align}

The estimates are similar for $n$ odd. The lemma is proved.
\end{proof}

We will prove in Appendix~\ref{app0918} the following lemma. 

\begin{lemma}\label{lem7}
Let~$C$ be an arbitrary positive real number. 
For arbitrary $l\in\ZZ^{\ge2}$, we have 
\begin{equation}\label{lem7formula1}
\lim_{g\to+\infty}\max_{2\leq n\leq C\log(g)}\max_{d_1,\dots,d_n\ge 0 \atop d_1+\cdots+d_n=3g+n-3 }\sum_{\sigma\in S_{n,1}^{\{2l\}}}\left|\gamma_{\underline{d},\sigma}(g)\right|=0.
\end{equation}
Moreover, there exists a constant $C_{12}=C_{12}(C)$, such that
\begin{equation}\label{lem7formula2}
\max_{2\leq n\leq C\log(g)} \max_{3\leq l\leq \left[\frac{n}2\right]}
\max_{d_1,\dots,d_n\ge 0 \atop d_1+\cdots+d_n=3g+n-3 }\max_{\sigma\in S_{n,1}^{\{2l\}}}\frac{2^{n-2}\left|\gamma_{\underline{d},\sigma}(g)\right|g^{l-3}}{\left(\frac{\pi^2}6\right)^n14^n} < C_{12}.
\end{equation}

\end{lemma}
	
We are ready to prove Theorem~\ref{thm1}.

\begin{proof}[Proof of Theorem~\ref{thm1}]

For the case $n=1$ we know from~\eqref{f1mr} that $\G_{3g-2}(g)-1=0$.
Now we consider the case $n\in \ZZ^{\ge 2}$. Take $E=\bigl[50(C+3)^2\bigr]$. (One can verify that for an arbitrary $l\in\mathbb{Z}^{\ge E}$, $l$ satisfies that $l>4+C\log(14\pi^2l/3)$.)
By using \eqref{estimatesn12l}, \eqref{formulaWK}, \eqref{defepsilon}
and the triangle inequality, we obtain that 
\begin{align}\label{thm1proof1}
&\max_{{2\leq n\leq C\log(g)}} \max_{d_1,\dots,d_n\ge0\atop d_1+\cdots+d_n=3g-3+n} 
\left|\G_{d_1,\dots,d_n}(g)-1\right|  \\
&\leq \max_{{2\leq n\leq C\log(g)}} \max_{d_1,\dots,d_n\ge0\atop d_1+\cdots+d_n=3g-3+n} \nn\\
& \Biggl( \biggl|\sum_{\sigma\in S_{n,1}^{\{2\}}} \gamma_{\underline{d},\sigma}(g)-1\biggr|+\sum_{l=2}^{E}\sum_{\sigma\in S_{n,1}^{\{2l\}}} |\gamma_{\underline{d},\sigma}(g)|+\sum_{l=E+1}^{\left[\frac{n}2\right]}\sum_{\sigma\in S_{n,1}^{\{2l\}}} |\gamma_{\underline{d},\sigma}(g)| \Biggr)\nn\\
& \leq \max_{{2\leq n\leq C\log(g)}} \max_{d_1,\dots,d_n\ge0\atop d_1+\cdots+d_n=3g-3+n} \nn\\
& \Biggl( \biggl|\sum_{\sigma\in S_{n,1}^{\{2\}}} \gamma_{\underline{d},\sigma}(g)-1\biggr|+\sum_{l=2}^{E}\sum_{\sigma\in S_{n,1}^{\{2l\}}} |\gamma_{\underline{d},\sigma}(g)| \nn\\
& +\sum_{l=E+1}^{\left[\frac{n}2\right]}(2l)^{n-1}\max_{\sigma\in S_{n,1}^{\{2l\}}} |\gamma_{\underline{d},\sigma}(g)| \Biggr).\nn
\end{align}

We are going to estimate the right-hand side of~\eqref{thm1proof1} term by term for~$g$ large. 
By Lemma~\ref{lem6} we know that the first term of the right-hand side of~\eqref{thm1proof1} 
tends to~0 as $g\to \infty$. 
From formula~\eqref{lem7formula1}
of Lemma~8, one can deduce that
\begin{align}
\lim_{g\to \infty}\max_{{2\leq n\leq C\log(g)}}\max_{d_1,\dots,d_n\ge0 \atop d_1+\cdots+d_n=3g+n-3}\sum_{l=2}^{E}\sum_{\sigma\in S_{n,1}^{\{2l\}}} |\gamma_{\underline{d},\sigma}(g)|=0.
\end{align}
For the third term of the right-hand side of~\eqref{thm1proof1}, using formula~\eqref{lem7formula2} of Lemma~\ref{lem7}, we have 
\begin{align}
&\max_{{2\leq n\leq C\log(g)}} \max_{d_1,\dots,d_n\ge0\atop d_1+\cdots+d_n=3g-3+n}\sum_{l=E+1}^{\left[\frac n2\right]} \, (2l)^{n-1} \max_{\sigma\in S_{n,1}^{\{2l\}}} |\gamma_{\underline{d},\sigma}(g)| \label{thm1proof4}\\
\leq &\max_{{2\leq n\leq C\log(g)}} \max_{d_1,\dots,d_n\ge0\atop d_1+\cdots+d_n=3g-3+n}\sum_{l=E+1}^{\left[\frac n2\right]} \, (2l)^{n-1} \, C_{12} \, \frac{\left(\frac{\pi^2}6\right)^n14^n}{2^{n-2}g^{l-3}}\nn\\
%\leq &\max_{{2\leq n\leq C\log(g)}} \max_{d_1,\dots,d_n\ge0\atop d_1+\cdots+d_n=3g-3+n}\sum_{l=E+1}^{\left[\frac n2\right]} 4 \, \binom{n-1}{2l-1} \, C_{12} \, \frac{\left(\frac{7\pi^2}3l\right)^{n}}{g^{l-3}}\nn\\
\leq &\max_{{2\leq n\leq C\log(g)}} \max_{d_1,\dots,d_n\ge0\atop d_1+\cdots+d_n=3g-3+n}\sum_{l=E+1}^{\left[\frac n2\right]} 4 \, C_{12} \, \frac{g^{C\log\left(\frac{7\pi^2}3l\right)}}{g^{l-3}}\nn\\
\leq & \; 4 \, C_{12} \, \max_{{2\leq n\leq C\log(g)}} \max_{d_1,\dots,d_n\ge0\atop d_1+\cdots+d_n=3g-3+n} g^{-1-C\log2}\to0\quad (g\to\infty).\nn
\end{align}
Theorem~\ref{thm1} is proved. 
\end{proof}
\begin{remark}
As an application of Theorem~\ref{thm1}, one can improve a result of Liu and Xu~\cite{LX} on the asymptotics of Weil--Petersson volumes (see the Theorems 1.2, 3.8 in~\cite{LX}). 
Indeed, from Theorem~\ref{thm1} one can deduce that for any fixed $n, k\ge1$ being integers, 
\begin{equation}
\lim_{g\to\infty}\max_{d_1,\dots,d_n\ge 0 \atop d_1+\cdots+d_n=3g+n-3-k}\left|\frac{24^g\,g!\,5^k\prod_{j=1}^n(2d_j+1)!!}{(6g+2n+2k-5)!!}\int_{\Mgnbar}\kappa_{1}^k\psi_1^{d_1}\cdots\psi_n^{d_n }-1\right|=0.
\end{equation}
\end{remark}

\section{Normalized intersection numbers $\G_{d_1,\dots,d_n}(g)$ 
with fixed $n,d_1,\dots,d_{n-1}$} \label{section3}

In this section, we prove Theorem~\ref{thm2} and the validity of $k=1,2$ parts of Conjecture~\ref{conj1011}.
\begin{proof}[Proof of Theorem~\ref{thm2}]
For $n=1$, the statements of the theorem easily follow from~\eqref{f1mr}. 
Now consider the case $n\geq2$. 
For convenience, denote $d_n=3g-3+n-|d|$. 
Due to cyclic symmetry we could take $\sigma\in S_n$ with 
$\sigma(n)=n$ in the sum in~\eqref{formulaWK}. 
For such~$\sigma$ one can show that for arbitrary~$g$ the number of the elements satisfying the constraints
\begin{equation}\label{K-1}
K_{\underline{d}, \sigma, q} (\underline{j})\ge-1, \quad q=1,\dots,n 
\end{equation}
is less than or equal to 
$(|d|+n)^n$ (due to arguments similar to the ones given after~\eqref{defepsilon}). Using~\eqref{a} we see that 
 each possibly nonzero summand in the right-hand side of~\eqref{defepsilon}  
 after multiplying by the factor 
 $\frac{24^g g! (-1)^{m(\sigma)+1}}{(6g-5+2n)!!}$ is a rational function of~$g$. 
 Since $S_n$ is a finite set we therefore conclude the existence of a rational function $R(g;d_1,\dots,d_{n-1})$, such that 
 $\G_{d_1,\dots,d_n}(g)=R(g;d_1,\dots,d_{n-1})$. When $d_1,\dots,d_{n-1}$ are all greater than or equal to~1, the statement on the positions of all possible poles 
 of $R(g;d_1,\dots,d_{n-1})$
 could be proved again using~\eqref{a}. 
 In general, if $d_1,\dots,d_{n-1}$ contain zeros, the statement on the positions of 
 possible poles can be proved by mathematical induction with the further application of the string equation~\eqref{stringeq}.
The statement that the leading term $\G_0$ in~\eqref{37} is identically~1 is a consequence of Theorem~\ref{thm1}. 
The full asymptotic behavior~\eqref{37} is then implied by the rationality. 
The theorem is proved. 
\end{proof}

\begin{remark}
In the above proof of Theorem~\ref{thm2} we used Theorem~\ref{thm1} to get the 
leading term~$\G_0$, but actually, 
when $n$ is fixed like here in Theorem~\ref{thm2}, to obtain 
$\G_0$ from~\eqref{f1mr}--\eqref{f2mr} (or equivalently from Proposition~\ref{prop1}) 
the estimates could be given in a much easier procedure, with  
several key observations (like the role of unimodal permutations) 
in the proof of Theorem~\ref{thm1} kept. 
\end{remark}

Let us now indicate another proof of Theorem~\ref{thm2} based on the work of Liu and Xu~\cite{LX}. 
Following Liu and Xu, define
\begin{align}
& \C_{d_1,\dots,d_{n-1}}(g) = 
\frac{24^gg!\langle\tau_{d_1}\cdots\tau_{d_{n-1}}\tau_{3g-3+n-|d|}\rangle\prod_{j=1}^{n-1}(2d_j+1)!!}{(6g)^{|d|}}, \label{defC} 
\end{align}
and define
\begin{equation}\label{defP}
\cP_{d_1,\dots,d_{n-1}}(g) = (6g)^{|d|} \, \C_{d_1,\dots,d_{n-1}}(g).
\end{equation}
Here, $n\geq 1$, $d_1,\dots,d_{n-1}\geq 0$, and $|d|:=\sum_{j=1}^{n-1}d_j$. 
Liu and Xu proved the following two statements by using Virasoro constraints~\eqref{vira}.

\smallskip
\noindent {\bf Theorem~A} (Liu--Xu~\cite[Corollary~3.3]{LX}). 
{\it For any fixed $n\geq 1$ and fixed $d_1,\dots,d_{n-1}\geq 0$,  
the number $\C_{d_1,\dots,d_{n-1}}(g)$ has the asymptotic expansion:
\begin{equation} \label{LX1}
\C_{d_1,\dots,d_{n-1}}(g) \sim \sum_{k\ge0} \frac{\C_{k}}{g^k} \quad  (g\to\infty),
\end{equation}
where $\C_0=1$ and  
$\C_k=\C_k(d_1,\dots,d_{n-1})$, $k\ge1$, are constants.
Moreover, the right-hand side of~\eqref{LX1} truncates to a finite sum: 
$\C_k\equiv0$ whenever $k>|d|$.  }
	
The truncation property given in Liu--Xu's Theorem~A is 
even more precisely stated in Liu--Xu's next theorem.

\smallskip

\noindent {\bf Theorem~B} (Liu--Xu~\cite[Theorem~4.1]{LX}). 
{\it For any fixed $n\ge1$ and fixed $d_1,\dots ,d_{n-1}\ge0$, there exits a polynomial $P_{d_1,\dots ,d_{n-1}}(g)\in \mathbb{Q}[g]$ of deg~$|d|$ with the highest-degree term $(6g)^{|d|}$ and the constant term $\prod_{\ell=1}^{|d|}(n-\ell-2)\prod_{j=1}^{n-1}\frac{(2d_j+1)!!}{d_j!}$, such that 
$$\cP_{d_1,\dots,d_{n-1}}(g)=P_{d_1,\dots,d_{n-1}}(g), \quad  \forall\, g\in \mathbb{Z}_{\ge 0}.$$ 
Moreover, $P_{d_1,\dots,d_{n-1}}(g)\in\mathbb{Z}$ $($$\forall\, g\in\mathbb{Z}$$)$, and $2^{[|d|/3]}P_{d_1,\dots ,d_{n-1}}(g)\in\mathbb{Z}[g]$.}

Theorem~\ref{thm2} easily follows from Theorem~B. 

By using the Virasoro constraints~\eqref{vira}, 
Liu and Xu~\cite{LX} obtained the first two 
 explicit expressions for~$\C_k$ as follows:

{\small \begin{align}
\C_1&=-\frac{|d|^2}{6}+\frac{(n-1)|d|}{3}+\frac{n^2-5n}{12}+\frac{5p_0-p_0^2}{12}, \label{C1expression}\\
\C_2&=\frac{|d|^4}{72}-\frac{(3n-2)|d|^3}{54}+\frac{n(3n+1)|d|^2}{72}  \label{C2expression} \\
		&+\frac{(6n^3-48n^2+54n-11)|d|}{216}+\frac{n(3n^3-50n^2+189n+14)}{864}\nn\\
		&+\frac{p_0^2(4|d|^2-8n|d|+8|d|-2n^2+22n-12p_1+47)}{288}\nn\\
		&+\frac{p_0^4}{288}-\frac{23p_0^3}{432}-\frac{5p_2}{72}-\frac{17p_1}{72}\nn\\&
		+\frac{p_0(-30|d|^2+60n|d|-60|d|+15n^2-165n+90p_1-7)}{432}\nn,
\end{align}} 

\vspace{-2mm}

\noindent where $p_i$ denotes the multiplicity of~$i$ in~$(d_1,\dots,d_{n-1})$.
We notice here 
the appearance of $|d|=d_1+\dots+d_{n-1}$ in 
the expression of $\C_1$, $\C_2$. 
{\it Remarkably}, formulae of $\G_1$ and $\G_2$ (cf.~\eqref{g1expression}, \eqref{g2expression}, \eqref{GkequalsGk}) are much simpler than 
$\C_1$ and~$\C_2$. 
According to the definitions, $\G_{d_1,\dots,d_n}(g)$ is related to $\C_{d_1,\dots,d_n}(g)$ as follows:
\begin{equation}\label{defG}
\G_{d_1,\dots,d_n}(g) = \frac{(6g)^{|d|} \, \C_{d_1,\dots,d_{n-1}}(g)}{\prod_{\ell=1}^{|d|} (6g+2n-3-2\ell)}.  
\end{equation}
The just-mentioned remarkable simplification can then be  
more precise: multiplying the asymptotic expansion~\eqref{LX1} of $\C_{d_1,\dots,d_{n-1}}(g)$ by the factor 
$$
\frac{(6g)^{|d|}}{\prod_{\ell=1}^{|d|} (6g+2n-3-2\ell)} 
$$
gives the asymptotic expansion~\eqref{37} of $\G_{d_1,\dots,d_{n-1},3g-3+n-|d|}(g)$, that makes the appearance of~$|d|$ all disappear, conjecturally (Conjecture~\ref{conj1011}). 
For $k=1,2$, one can verify straightforwardly that 
$|d|'s$ do disappear in this way, and we get the expressions for $\G_1$ and~$\G_2$ 
fulfilling all the statements in Conjecture~\ref{conj1011} with $k=1,2$.
This proves the validity of the $k=1,2$ parts of Conjecture~\ref{conj1011}.
(One can also verify that the expressions for $\G_1$ and~$\G_2$ obtained in this way 
coincide with the right-hand sides of \eqref{g1expression}, \eqref{g2expression}.) 
However, we note that these do {\it not} imply the $K=1,2$ parts of Conjecture~\ref{conj1}.

Let us end this section by expressing the coefficients $\C_k$ and $\G_k$ in~\eqref{LX1} and~\eqref{37}
 in terms of coefficients of the polynomials 
$P_{d_1,\dots,d_{n-1}}(g)$. Write 
\begin{equation}\label{writepyals}
P_{d_1,\dots,d_{n-1}}(g)=:\sum_{m=0}^{|d|}\alpha_{m, d_1,\dots,d_{n-1}} g^{m},
\end{equation}
where $\alpha_{m,d_1,\dots,d_{n-1}}\in\mathbb{Q}$, $0\leq m \leq |d|$. 
Then we have
\begin{align}
& \C_{k} = \frac{\alpha_{|d|-k,d_1,\dots,d_{n-1}}}{6^{|d|}},\\
& \G_k = \sum_{l=0}^{k} \frac{(-1)^l}{6^{|d|+l}}\alpha_{|d|+l-k,d_1,\dots,d_{n-1}}
\sum_{\substack{j_1,\dots, j_{|d|} \geq 0\\j_1+\cdots+j_{|d|}=l}}\prod_{\ell=1}^{|d|}(2n-3-2\ell)^{j_{\ell}}.
\end{align}

\section{Proof of Theorem~\ref{thm3}}\label{section4}
In this section, we investigate the {\it uniform in~$\underline{d}$} asymptotic expansion of the normalized intersection numbers
 $\G(\underline{d};g)$ when $g$ is large. 
 
\begin{proof}[Proof of Theorem~\ref{thm3}]
For $n=1$, from \eqref{f1mr} we know that the statement~\eqref{thm3formula} is trivial. 

Now fix $n\geq2$ to be an integer. 
For the case that $n$ is even, 
let $g\ge0$ be an integer, and let $d_1,\dots,d_n\ge[3K/2]$ be integers
 satisfying $d_1+\cdots+d_n=3g-3+n$. We have
\begin{align}
\G_{d_1,\dots,d_n}(g)=\sum_{\sigma\in S_{n,1}^{\{2\}}}\gamma_{\underline{d},\sigma}(g)
+\sum_{l=2}^{\left[\frac{n}2\right]}\sum_{\sigma\in S_{n,1}^{\{2l\}}} \gamma_{\underline{d},\sigma}(g).
	\label{thm3proof1}
\end{align}
Here we recall that $\gamma_{\underline{d},\sigma}(g)$ is defined by~\eqref{defepsilon}.
	
We start with estimating the second term of the right-hand side of~\eqref{thm3proof1}.
For $l\geq2$ and for each $\sigma \in S_{n,1}^{\{2l\}}$, denote by $1\leq t_1<t_2<\cdots<t_{2l}\leq n$ the integers 
such that for any $1\leq u\leq l$,
\begin{align}
&\sigma(t_{2u-1})<\sigma(t_{2u-1}+1)<\cdots<\sigma(t_{2u}), \\
&\sigma(t_{2u})>\sigma(t_{2u}+1)>\cdots>\sigma(t_{2u+1}).
\end{align}
Since $d_1,\dots,d_n\ge[3K/2]$ and $l\ge2$, on the right-hand side of~\eqref{defepsilon} we only need to consider that the summation index vector~$\underline{j}$ belongs to the set:
\begin{align}
V_{\underline{d},\sigma,-1}^{3g+\frac{3n}2-8-\left[\frac{3K}2\right]} \bigsqcup V_{\underline{d},\sigma,3g+\frac{3n}2-7-\left[\frac{3K}2\right]}^{3g+\frac{3n}2-5-\left[\frac{3K}2\right]}.
\end{align}
Taking $m=5+[3K/2]$ in Lemma~\ref{lem3}, from an 
argument similar to that of~\eqref{lemma8formula1} one can deduce that there exists 
a constant $C_{14}$, that is 
{\it independent of $g,\sigma,d_1,\dots,d_n$}, such that 
\begin{align}\label{thm3formula1}
\left|\sum_{\underline{j}\in V_{\underline{d},\sigma,-1}^{3g+\frac{3n}2-8-\left[\frac{3K}2\right]}} \!\!\!\!\!\! a_{K_{\underline{d},\sigma,1}(\underline{j}),\dots,K_{\underline{d},\sigma,n}(\underline{j})} \right|  
\leq C_{14} \, \frac{(6g+3n-7)!! \, g^{-(K+1)}}{24^{g+\frac{n}2-1}\left(g+\frac{n}2-1\right)!} 
\end{align}
hold true for sufficiently large $g$.
One can show that for arbitrary $g$, the number of elements in 
$V_{\underline{d},\sigma,3g+3n/2-7-\left[3K/2\right]}^{3g+3n/2-5-\left[3K/2\right]}$ 
is smaller than or equal to
\beq\label{counting91}
\left(2+\left[\frac{3K}2\right]+\frac{n}2\right)^n+\left(3+\left[\frac{3K}2\right]+ \frac{n}2\right)^n+\left(4+\left[\frac{3K}2\right]+\frac{n}2\right)^n.
\eeq
It follows from Lemma~\ref{lem1even0906} with $m=2+[3K/2]$ that there exists a constant $C_{15}$, such that 
\begin{align}
&\max_{-1\leq k_1,\dots,k_n\leq 3g+\frac{3n}2-5-\left[\frac{3K}2\right]\atop k_1+\cdots+k_n=3g-3+n} 
\left|a_{k_1,\dots,k_n}\right| \leq C_{15} \, \frac{(6g+3n-7)!! \, g^{-(K+1)}}{24^{g+\frac{n}2-1}\left(g+\frac{n}2-1\right)!} 
\label{thm3formula3}
\end{align}
holds for sufficiently large $g$.
Since $n$ is fixed, using \eqref{thm3formula1}--\eqref{thm3formula3}
as well as the fact that $C_{14}$ is independent of $d_1,\dots,d_n$, 
we conclude that 
\beq
\lim_{g\to\infty} g^{K} \max_{d_1,\dots,d_n\geq \left[\frac{3K}2\right] \atop d_1+\cdots+d_n=3g-3+n} 
\left|\sum_{l=2}^{\left[\frac{n}2\right]}\sum_{\sigma\in S_{n,1}^{\{2l\}}} \gamma_{\underline{d},\sigma}(g)\right| =0.\label{formula0928}
\eeq

Let us now estimate the first term on the right-hand side of~\eqref{thm3proof1}.
For $\sigma\in S_{n,1}^{(r,n-r)}$, we decompose the set
$V_{\underline{d},\sigma,-1}^{3g+3n/2-3}$ as follows:
\begin{align}
&V_{\underline{d},\sigma,-1}^{3g+\frac{3n}2-8-\left[\frac{3K}2\right]} \bigsqcup 
V_{\underline{d},\sigma,3g+\frac{3n}2-7-\left[\frac{3K}2\right]}^{3g+\frac{3n}2-6-\left[\frac{3K+1}2\right]}\label{decompose1}\\
&\quad\bigsqcup
 V_{\underline{d},\sigma,3g+\frac{3n}2-5-\left[\frac{3K+1}2\right]}^{3g+\frac{3n}2-5-\left[\frac{3K}2\right]}\bigsqcup V_{\underline{d},\sigma,3g+\frac{3n}2-4-\left[\frac{3K}2\right]}^{3g+\frac{3n}2-3}.\nn
\end{align}

Similarly as above, we deduce 
by taking $m=5+[3K/2]$ in Lemma~\ref{lem3} that there exists 
a constant $C_{16}$, that is {\it independent of $g,\sigma,d_1,\dots,d_n$}, such that 
\begin{align}
& \left|\sum_{V_{\underline{d},\sigma,-1}^{3g+\frac{3n}2-8- \left[\frac{3K}2\right]}}a_{K_{\underline{d},\sigma,1}(\underline{j}),\dots,K_{\underline{d},\sigma,n}(\underline{j})}\right| \leq C_{16}\,\frac{(6g+3n-7)!!}{24^{g+\frac{n}2-1}\left(g+\frac{n}2-1\right)!}\frac{1}{g^{K+1}}. \label{thm3formula5}
\end{align}

Using a similar argument to~\eqref{lem6estimate5} we have 
\begin{align}\label{thm3formula7}
		&\left|(-1)^{m(\sigma)+1}\sum_{\underline{j}\in V_{\underline{d},\sigma,3g+\frac{3n}2-7-\left[\frac{3K}2\right]}^{3g+\frac{3n}2-6-\left[\frac{3K+1}2\right]}}a_{K_{\underline{d},\sigma,1}(\underline{j}),\dots,K_{\underline{d},\sigma,n}(\underline{j})}\right|\\
		&\leq\sum_{f=6+\left[\frac{3K+1}2\right]}^{7+\left[\frac{3K}2\right]}\left|d_1\, A(f,n)+B(f,n)\right|\max_{\substack{-1\le k_1,\dots,k_n\le 3g+\frac{3n}2-f \\ k_1+\cdots+k_n=3g-3+n}}a_{K_{\underline{d},\sigma,1}(\underline{j}),\dots,K_{\underline{d},\sigma,n}(\underline{j})}\nn\\
		&\leq\sum_{f=6+\left[\frac{3K+1}2\right]}^{7+\left[\frac{3K}2\right]}\left|d_1\,A(f,n)+B(f,n)\right|\frac{(6g+3n-7)!!}{24^{g+\frac{n}2-1} \bigl(g+\frac{n}2-1\bigr)!}\frac{C_{1}(f-3)}{g^{K+2}}. \nn
\end{align}
Here, $A(f,n)$ and $B(f,n)$ for every $f$ are certain polynomials of~$n$, 
and the validity of the second inequality can be deduced from 
Lemma~\ref{lem1even0906} with $m=f-3$.

To proceed let us notice the following fact: for $g$ being sufficiently large, for given $k_1,\dots,k_n\ge-1$ satisfying $k_1+\cdots+k_n=3g+n-3$
and $\max_{1\le j\le n} k_j\ge3g+3n/2-4-[3K/2]$, and for 
 given $d_1,\dots,d_n\ge[3K/2]$ satisfying $d_1+\cdots+d_n=3g+n-3$,
if $k_{r+1} \ge 3g+3n/2-4-[3K/2]$, then the equations
\begin{align}\label{equa1}
   K_{\underline{d},\sigma,q}(\underline{j})=k_q,\quad q=1,\dots,n
\end{align}
for $\underline{j}\in({\mathbb{Z}^{\ge0}})^n$ have exactly $d_1-k_1$ solutions; if $k_{r+1} \le 3g+3n/2-5-[3K/2]$ 
the above equations~\eqref{equa1} for $\underline{j}\in({\mathbb{Z}^{\ge0}})^n$ do not have solutions. Thus,
\begin{align}\label{thm3formula6}
&\sum_{r=1}^{n-1}\sum_{\sigma\in S_n^{(r,n-r)}} (-1)^{m(\sigma)+1}\sum_{\underline{j}\in V_{\underline{d},\sigma,3g+\frac{3n}2-4-\left[\frac{3K}2\right]}^{3g+\frac{3n}2-3}} a_{K_{\underline{d},\sigma,1}(\underline{j}),\dots,K_{\underline{d},\sigma,n}(\underline{j})}\\
=&\sum_{r=1}^{n-1}(-1)^{n-r+1}\sum_{\sigma\in S_n^{(r,n-r)}} 
\sum_{\substack{ k_1,\dots,k_{n}\ge-1\\ k_1+\cdots+k_n=3g+n-3 \\ k_{r+1}\ge3g+\frac{3n}2-4-\left[\frac{3K}2\right]}}(d_1-k_1)a_{k_1,\dots,k_n}\nn\\
=&\sum_{r=1}^{n-1}(-1)^{n-r+1}\binom{n-2}{r-1}
\sum_{\substack{ k_1,\dots,k_{n}\ge-1\\ k_1+\cdots+k_n=3g+n-3 \\ k_n\ge3g+\frac{3n}2-4-\left[\frac{3K}2\right]}}(d_1-k_{n-r})a_{k_1,\dots,k_n}\nn\\
=&\sum_{r=1}^{n-1}(-1)^{n-r}\binom{n-2}{r-1}\sum_{\substack{ k_1,\dots,k_{n}\ge-1\\ k_1+\cdots+k_n=3g+n-3\\  k_n\ge3g+\frac{3n}2-4-\left[\frac{3K}2\right]}}k_{n-r}a_{k_1,\dots,k_n}\nn\\
&+\delta_{n,2}\,d_1\sum_{k_1=-1}^{\left[\frac{3K}2\right]}a_{k_1,3g-1-k_1}.\nn
\end{align}
Here, in the second equality we used the cyclic symmetry of $a_{k_1,\dots,k_n}$.

We now divide the consideration into two cases: the $n\geq3$ case and the $n=2$ case.
For $n\geq3$, we have
\begin{align}
V_{\underline{d},\sigma,3g+\frac{3n}2-5-\left[\frac{3K+1}2\right]}^{3g+\frac{3n}2-5-\left[\frac{3K}2\right]}=
\left\{\begin{array}{ll}
V_{\underline{d},\sigma,3g+\frac{3n}2-5-\left[\frac{3K}2\right]}^{3g+\frac{3n}2-5-\left[\frac{3K}2\right]}, \quad &K\,{\rm even,}\\	V_{\underline{d},\sigma,3g+\frac{3n}2-5-\left[\frac{3K}2\right]}^{3g+\frac{3n}2-5-\left[\frac{3K}2\right]}\bigsqcup V_{\underline{d},\sigma,3g+\frac{3n}2-6-\left[\frac{3K}2\right]}^{3g+\frac{3n}2-6-\left[\frac{3K}2\right]}, \quad &K\,{\rm odd.}
\end{array}\right.\label{decompose2}
\end{align}
For simplicity we assume that $K$ is even ($K$ odd similar), and we have
\begin{align}
&\sum_{r=1}^{n-1}\sum_{\sigma\in S_n^{(r,n-r)}}(-1)^{m(\sigma)+1}\sum_{\underline{j}\in V_{\underline{d},\sigma,3g+\frac{3n}2-5-\left[\frac{3K}2\right]}^{3g+\frac{3n}2-5-\left[\frac{3K}2\right]}}a_{K_{\underline{d},\sigma,1}(\underline{j}),\dots,K_{\underline{d},\sigma,n}(\underline{j})}\label{Thm3proof10}\\
&=\sum_{r=1}^{n-1}\sum_{\sigma\in S_n^{(r,n-r)}}(-1)^{n-r+1} \nn\\
&\times\left(\sum_{\substack{k_1,\dots,k_n\ge-1\\k_1+\cdots+k_n=3g+n-3\\k_{r+1}=\left[\frac{3K}2\right]+1\\\max\{k_q\}=3g+\frac{3n}2-5-[\frac{3K}2]}}+\sum_{\substack{k_1,\dots,k_n\ge-1\\k_1+\cdots+k_n=3g+n-3\\k_{r+1}=3g+\frac{3n}2-5-\left[\frac{3K}2\right]}}\right) \omega_{\underline{d},\sigma,\underline{k}} \, a_{k_1,\dots,k_n},\nn
\end{align}
where $\omega_{\underline{d},\sigma,\underline{k}}$ denotes the number of the solutions to the equations
$$K_{\underline{d},\sigma,q}(\underline{j})=k_q,\quad q=1,\dots,n$$
for $\underline{j}\in\left(\ZZ^{\ge0}\right)^n$.
Note that for $k_{r+1}=\left[3K/2\right]+1$, we have
\begin{align}
0\leq&\omega_{\underline{d},\sigma,\underline{k}}\leq 1,\label{omega1}
\end{align}
and that for $k_{r+1}=3g+\frac{3n}2-5-\left[3K/2\right]$,
\begin{align}
d_1-2-k_1\leq&\omega_{\underline{d},\sigma,\underline{k}}\leq d_1-k_1.\label{omega2}
\end{align} 
By using \eqref{Thm3proof10}, \eqref{omega1}, \eqref{omega2}, 
one can deduce from Lemma~\ref{lem1even0906} 
that 
\begin{align}
& \left|\sum_{r=1}^{n-1}\sum_{\sigma\in S_n^{(r,n-r)}}(-1)^{m(\sigma)+1}\sum_{\underline{j}\in V_{\underline{d},\sigma,3g+\frac{3n}2-5-\left[\frac{3K}2\right]}^{3g+\frac{3n}2-5-\left[\frac{3K}2\right]}}a_{K_{\underline{d},\sigma,1}(\underline{j}),\dots,K_{\underline{d},\sigma,n}(\underline{j})}\right|\label{Thm3proof5}\\
& \leq C_{17} \, \frac{(6g+3n-7)!!}{24^{g+\frac{n}2-1}\left(g+\frac{n}2-1\right)!}\frac{1}{g^{K+1}},\nn
\end{align}
for some constant $C_{17}$, that is independent of $g,d_1,\dots,d_n$.
For $n=2$, one can verify that
\begin{align}
&\lim_{g\to \infty}g^K \Biggl|(-1)^{m({\rm id})+1} 
\sum_{\underline{j} \in 
V_{\underline{d}, {\rm id},3g-2-\left[\frac{3K+1}2\right]}^{3g-2-\left[\frac{3K}2\right]}}
a_{d_1-1-j_1-j_2,d_2+1+j_1+j_2} \label{Thm3proof6}\\
& \qquad\qquad\qquad\qquad\quad -d_1\sum_{k_1=\left[\frac{3K}2\right]+1}^{\left[\frac{3K+1}2\right]+1}a_{k_1,3g-1-k_1}\Biggr|=0.\nn
\end{align}

We conclude from
 \eqref{formula0928}, \eqref{thm3formula5}, \eqref{thm3formula7}, \eqref{thm3formula6}, \eqref{Thm3proof5}, 
 \eqref{Thm3proof6}
that for every $n\ge2$, 
\begin{align}\label{Thm3proof4}
&\lim_{g\to\infty}\max_{\substack{d_1,\dots,d_n\ge \left[\frac{3K}2\right] \\ d_1+\cdots+d_n=3g+n-3}} 
g^K \, \Bigg|\G_{d_1,\dots,d_n}(g)\\
&-\frac{24^gg!}{(6g+2n-5)!!}\sum_{r=1}^{n-1}(-1)^{n-r}\binom{n-2}{r-1} \!\! 
   \sum_{ k_1,\dots,k_{n}\ge-1\atop k_n\ge3g+\frac{3n}2-4-\left[\frac{3K}2\right]} \!\! k_{n-r}a_{k_1,\dots,k_n}\Bigg| =0.\nn
\end{align}
Here for the case $n\geq3$ we used again the elementary fact~\eqref{21}, and
 for the case $n=2$, we also used the following elementary formula:
\begin{align}
\lim_{g\to\infty}\max_{\substack{d_1,d_2\ge \left[\frac{3K}2\right] \\ d_1+d_2=3g-1}} g^K\,\frac{24^gg!}{(6g-1)!!}d_1\sum_{k_1=-1}^{\left[\frac{3K+1}2\right]+1}a_{k_1,3g-1-k_1}=0,
\end{align}
which can be deduced from the following identity derived in~\cite{Guo}:
\begin{align}
&\frac{24^gg!}{(6g-1)!!}\sum_{l=0}^{k+1}a_{l-1,3g-l}=\\
&\quad\quad\quad\frac{(6g-3-2k)!!}{(6g-1)!!} \, \left\{\begin{array}{ll}
		\frac{(6j-1)!!(g-1)!}{j!(g-j)!}(g-2j),  &k=3j-1,\\
		-2\frac{(6j+1)!!(g-1)!}{j!(g-1-j)!},  &k=3j,\\
		2\frac{(6j+3)!!(g-1)!}{j!(g-1-j)!},  &k=3j+1.
	\end{array}
	\right.\nn
\end{align}

From~\eqref{a} we know that for $k_1,\dots,k_n\geq-1$ with $k_n \ge 3g+3n/2-4-[3K/2]$ satisfying $k_1+\cdots+k_n=3g-3+n$,  $\frac{24^gg!}{(6g+2n-5)!!} a_{k_1,\dots,k_n}$ is a rational function
of~$g$. 
The estimates are similar when $n$ is odd. Combined with Theorem~\ref{thm1}, Theorem~\ref{thm3} is proved.
\end{proof}

Denote 
\begin{align}
& G(n,p_0,p_1,\dots) := 
\sum_{k\geq0} \frac{G_k\Bigl(n,p_0,\dots,p_{\left[\frac32 k\right]-1}\Bigr)}{g^k}, \label{1011112}\\
& G^{\rm app}_{K}\Bigl(n,p_0,p_1,\dots,p_{\left[\frac32 K\right]-1}\Bigr) := 
\sum_{k=0}^K \frac{G_k\Bigl(n,p_0,\dots,p_{\left[\frac32 k\right]-1}\Bigr)}{g^k}, \label{1011113}
\end{align}
where $K\geq 0$.
The following two corollaries can then be obtained by using the dilaton equation~\eqref{dilatoneq} and the string equation~\eqref{stringeq}, respectively.
\begin{cor}\label{cordilaton}
Assuming Conjecture~\ref{conj1011} is true, we have 
\begin{align}\label{1011114}
&G(n,p_0,p_1,\dots) = \frac{6g+3n-9}{6g+2n-5} \, G(n-1,p_0,p_1-1,p_2,\dots).
\end{align}
\end{cor}
Formula~\eqref{1011114} can be alternatively given by
$$
G_K^{app}(n,p_0,p_1,\dots) = \frac{6g+3n-9}{6g+2n-5} \, G_K^{app}(n-1,p_0,p_1-1,p_2,\dots)+\mathcal{O}\big(g^{-K-1}\big).
$$

\begin{cor}\label{corstring}
Assuming Conjecture~\ref{conj1011} is true, we have for every $K\geq1$,
\begin{align}%\label{stringG}
& G(n,p_0,p_1,\dots) = \frac1{6g+2n-5} \label{stringapp} \\
& \times \biggl(
3 p_1 \Bigl(G(n-1,p_0,p_1-1,p_2,\dots) 
-G(n-1,p_0-1,p_1,\dots,)\Bigr) \nn\\
&+ \sum_{i=2}^{\infty} (2i+1) \, p_i \, \Bigl( -G (n-1,p_0-1,p_1,\dots) \nn\\
& \quad + G(n-1,p_0-1,p_1,\dots,p_{i-2},p_{i-1}+1,p_i-1,p_{i+1},\dots,)\Bigr) \nn\\
& +(6g+3n-6-p_0) \, G(n-1,p_0-1,p_1,\dots)\biggr). \nn
\end{align}
\end{cor}
Formula~\eqref{stringapp} can be equivalently written as
\begin{align}
& G^{\rm app}_K\Bigl(n,p_0,\dots,p_{\left[\frac32 K\right]-1}\Bigr) = \frac1{6g+2n-5} \nn \\
& \times \biggl(
3 p_1 \Bigl(G^{\rm app}_K \Bigl(n-1,p_0,p_1-1,p_2,\dots,p_{\left[\frac32 K\right]-1}\Bigr) \nn\\
& \qquad\qquad 
-G^{\rm app}_K \Bigl(n-1,p_0-1,p_1,\dots,p_{\left[\frac32 K\right]-1}\Bigr)\Bigr) \nn\\
&+ \sum_{i=2}^{\left[\frac{3K}{2}\right]} (2i+1) \, p_i \, \Bigl( -G^{\rm app}_K \Bigl(n-1,p_0-1,p_1,\dots,p_{\left[\frac32 K\right]-1}\Bigr) \nn\\
& \quad + G^{\rm app}_K\Bigl(n-1,p_0-1,p_1,\dots,p_{i-2},p_{i-1}+1,p_i-1,p_{i+1},\dots,p_{\left[\frac32 K\right]-1}\Bigr)\Bigr) \nn\\
& +(6g+3n-6-p_0) \, G^{\rm app}_K\Bigl(n-1,p_0-1,p_1,\dots,p_{\left[\frac32 K\right]-1}\Bigr)\biggr) + 
\mathcal{O}\bigl(g^{-K-1}\bigr). \nn
\end{align}
Substituting~\eqref{1011112} in~\eqref{1011114} we obtain 
\begin{align}
&G_k\Bigl(n,p_0,\dots,p_{\left[\frac{3k}{2}\right]-1}\Bigr) -
G_k\Bigl(n-1,p_0,p_1-1,p_2,\dots,p_{\left[\frac{3k}{2}\right]-1}\Bigr)	\label{dilaton}\\
& = \sum_{j=0}^{k-1} (-1)^j \, 
\frac{(n-4)(2n-5)^j}{6^{j+1}} \, G_{k-1-j}\Bigl(n-1,p_0,p_1-1,p_2,\dots,p_{\left[\frac{3}{2}(k-1-j)\right]-1}\Bigr).\nn
\end{align}
Similarly, substituting~\eqref{1011113} in~\eqref{stringapp} we obtain
\begin{align}
&G_k\Bigl(n,p_0,\dots,p_{\left[\frac{3k}{2}\right]-1}\Bigr) - G_{k}\Bigl(n-1,p_0-1,p_1,\dots,p_{\left[\frac{3k}2 \right]-1}\Bigr)\label{string}\\
&=3p_1 \sum_{j=0}^{k-1} (-1)^j \, \frac{(2n-5)^j}{6^{j+1}} \, \Bigl(G_{k-1-j} \Bigl(n-1,p_0,p_1-1,p_2,\dots,p_{\left[\frac32 (k-1-j)\right]-1}\Bigr) \nn\\
& \qquad\qquad -G_{k-1-j} \Bigl(n-1,p_0-1,p_1,\dots,p_{\left[\frac32 (k-1-j)\right]-1}\Bigr)
\Bigr)\nn\\
& + \sum_{j=0}^{k-1} (-1)^j \, \frac{(2n-5)^j}{6^{j+1}} \, \sum_{i=2}^{\left[\frac{3k-3-3j}{2}\right]}  (2i+1)p_i \nn\\
& \Bigl(G_{k-1-j}\Bigl(n-1,p_0-1,p_1,\dots,p_{i-2},p_{i-1}+1,p_i-1,p_{i+1},\dots,p_{\left[\frac32 (k-1-j)\right]-1}\Bigr) \nn\\
& \quad - G_{k-1-j} \Bigl(n-1,p_0-1,p_1,\dots,p_{\left[\frac32 (k-1-j)\right]-1}\Bigr)\Bigr)
\nn\\
& + 6 \sum_{j=1}^{k} (-1)^j \, \frac{(2n-5)^j}{6^{j+1}} \, G_{k-j}\Bigl(n-1,p_0-1,p_1,\dots,p_{\left[\frac32 (k-j)\right]-1}\Bigr) \nn\\
& + (3n-6-p_0) \sum_{j=0}^{k-1} (-1)^j \, \frac{(2n-5)^j}{6^{j+1}} \nn\\
& \qquad \times G_{k-1-j} 
\Bigl(n-1,p_0-1,p_1,\dots,p_{\left[\frac32 (k-1-j)\right]-1}\Bigr). \nn
\end{align}

\begin{appendix}

\section{Proofs of the identity~\eqref{defia} and Lemmas \ref{lem1even0906}--\ref{lem3'}, \ref{lem7}} \label{app0918}

In this appendix, we give the proofs of the identity~\eqref{defia} and Lemmas \ref{lem1even0906}--\ref{lem3'},~\ref{lem7}. 

\begin{proof}[Proof of the identity~\eqref{defia}]
Let $\mu_1=\begin{pmatrix}0&1\\0&0\end{pmatrix}$, $\mu_2=\begin{pmatrix}0&0\\1&0\end{pmatrix}$, $\mu_3=\begin{pmatrix}1&0\\0&-1\end{pmatrix}$. Then $A_{3g}=b_{3g}\mu_1$, $A_{3g-1}=b_{3g-1}\mu_2$, $A_{3g-2}=-b_{3g-2}\mu_3$.  The following multiplication table could be helpful: 
$\mu_1^2=\mu_2^2=0$, $\mu_3^2=I$, $\mu_1\mu_2=\mu_4$, $\mu_2\mu_1=\mu_5$, $\mu_1\mu_3=\mu_1$, $\mu_3\mu_1=-\mu_1$, $\mu_2\mu_3=-\mu_1$, $\mu_3\mu_2=\mu_2$, where $\mu_4=\begin{pmatrix}1&0\\0&0\end{pmatrix}$, $\mu_5=\begin{pmatrix}0&0\\0&1\end{pmatrix}$. To prove~\eqref{defia}, let us 
prove the following identity:
\begin{align}\label{Ak1Akn}
A_{k_1}\cdots A_{k_n}=\left\{\begin{array}{ll}
b_{k_1}\cdots b_{k_n}I,\quad &\quad P_1,\\
(-1)^{\sum_{j=0}^{s}(i_{2j+1}-i_{2j}-1)}b_{k_1}\cdots b_{k_n}\mu_4,\quad  &\quad P_2, \\
(-1)^{\sum_{j=1}^{s}(i_{2j}-i_{2j-1}-1)}b_{k_1}\cdots b_{k_n}\mu_5,\quad  &\quad P_3, \\
-b_{k_1}\cdots b_{k_n}\mu_3, \quad &\quad P_4,\\
(-1)^{\sum_{j=0}^{s}(i_{2j+1}-i_{2j}-1)}b_{k_1}\cdots b_{k_n}\mu_1, \quad &\quad P_5,\\
(-1)^{\sum_{j=1}^{s+1}(i_{2j}-i_{2j-1}-1)}b_{k_1}\cdots b_{k_n}\mu_2,\quad  &\quad P_6,\\
0, \quad & \quad P_7.
\end{array}\right.
\end{align}
Here, $P_1$--$P_3$ are defined in Section~\ref{section2}, the conditions $P_4$--$P_6$ are defined by
\begin{itemize}
\item[$P_4$:] $n$ is odd and $k_i\equiv 1\,({\rm mod}\,3)$ for all $i=1,\dots,n$,
\item[$P_5$:] there exist $i_1<\dots<i_{2s+1}$, such that $k_{i_{2j-1}}\equiv0\,({\rm mod}\,3)$, $\forall\, j=1,\dots,s+1$, $k_{i_{2j}}\equiv 2\,({\rm mod}\,3)$, $\forall\, j=1,\dots,s$, 
and $k_{t}\equiv 1\,({\rm mod}\,3)$, $\forall \, t\notin \{{i_1}, \dots ,i_{2s+1}\}$, setting $i_0=0$,
\item[$P_6$:] there exist $i_1<\dots< i_{2s+1}$, such that $k_{i_{2j-1}}\equiv2\,({\rm mod}\,3)$, $\forall\, j=1,\dots,s+1$, $k_{i_{2j}}\equiv 0\,({\rm mod}\,3)$, $\forall\, j=1,\dots,s$, 
and $k_{t}\equiv 1\,({\rm mod}\,3)$, $\forall\,t\notin \{{i_1},\dots ,i_{2s+1}\}$, setting $i_{2s+2}=n+1$,
\end{itemize}
and $P_7$ denotes the complement of the union of conditions $P_1$--$P_6$. For $n=1$, the identity~\eqref{Ak1Akn} can be verified easily. Assuming~\eqref{Ak1Akn} holds for $n=N$, let us prove the validity of~\eqref{Ak1Akn} for $n=N+1$. There are 21 cases for the inductive procedure, which can be seen from Table~\ref{table123}.
\begin{table}[h!]
\begin{tabular}{|c|c|c|c|c|c|c|c|}
\hline
\diagbox{$k_{N+1}$ mod 3}{$(k_1,\dots,k_{N+1})$}{$(k_1,\dots,k_{N})$} & $P_1$ & $P_2$ & $P_3$ & $P_4$ & $P_5$ & $P_6$ & $P_7$\\
\hline
0 & $P_5$ & $P_5$ & $P_7$ & $P_5$ & $P_7$ & $P_3$ & $P_7$\\
\hline
1 & $P_4$ & $P_2$ &$P_3$ & $P_1$ & $P_5$ & $P_6$ & $P_7$\\
\hline
2 & $P_6$ & $P_7$ & $P_6$ & $ P_6$ & $P_2$ & $P_7$ & $P_7$\\
\hline
\end{tabular}\vspace{2mm}
\caption{The 21 cases in the proof of the identity~\eqref{Ak1Akn}.}\label{table123}
\end{table}
By the one-by-one verifications and by induction we obtain the validity of the identity~\eqref{Ak1Akn}. 
For example, consider the case that $(k_1,\dots,k_N)$ satisfies $P_2$ and $k_{N+1}\equiv 0\,({\rm mod}\,3)$. 
Then $(k_1,\dots,k_{N+1})$ satisfies~$P_5$. We have
\begin{align}
A_{k_1}\cdots A_{k_{N+1}}&=(-1)^{\sum_{j=0}^{s}(i_{2j+1}-i_{2j}-1)}b_{k_1}\cdots b_{k_N}\mu_4b_{k_{N+1}}\mu_1\nn\\
&=(-1)^{\sum_{j=0}^{s}(i_{2j+1}-i_{2j}-1)}b_{k_1}\cdots b_{k_{N+1}}\mu_1,\nn
\end{align}
which agrees with~\eqref{Ak1Akn}. Finally, \eqref{Ak1Akn} implies~\eqref{defia}.
\end{proof}

\begin{proof}[Proof of Lemma~\ref{lem1even0906}]
Let $n\in \mathbb{Z}_{\rm even}^{\geq2}$ and let $k_1,\dots,k_n$ be integers satisfying that 
$-1 \le k_1,\dots,k_n\le 3g+3n/2-3-m$ and $k_1+\cdots+k_n=3g-3+n$.
We assume that $a_{k_1,\dots,k_n}\neq0$ (otherwise trivial).
Using~\eqref{a} we find that  
\begin{align}\label{abinequality}
\left|a_{k_1,\dots, k_{n}}\right|\leq 2\left|b_{k_1}\cdots b_{k_{n}}\right|.
\end{align}
Without loss of generality we assume that 
$$k_n\geq k_{n-1} \geq \max\{k_1,\dots,k_{n-2}\}.$$
Define the numbers $s_{-1}, s_{0}, s_{1}, s_{2}$ as follows:
\begin{align}
	 s_{-1} & :={\rm card} \{i\in\{1,\dots,{n-1}\}\mid k_i=-1\},\\ 
	 s_{0} & :={\rm card}\{i\in\{1,\dots,{n-1}\}\mid k_i\equiv0\,({\rm mod}\,3)\},\\ 
	 s_{1} & :={\rm card} \{i\in\{1,\dots,{n-1}\}\mid k_i\equiv1\,({\rm mod}\,3)\},\\ 
	 s_2 & :={\rm card} \{i\in\{1,\dots,{n-1}\}\mid k_i\geq 2, k_i\equiv2\,({\rm mod}\,3)\}. 
\end{align}
Also denote $\tilde s_{-1}=s_{-1}-\delta_{k_{n-1},-1}$, $\tilde s_0=s_0-\delta_{k_{n-1}\equiv 0\,({\rm mod}\,3)}$, 
$\tilde s_1=s_1-\delta_{k_{n-1}\equiv 1\,({\rm mod}\,3)}$, $\tilde s_2=s_2-\delta_{k_{n-1}\equiv 2\,({\rm mod}\,3)}\delta_{k_{n-1}\ge2}$.
Denote 
$$\tilde q=m-\frac{n}2+\tilde s_{-1}- \tilde s_1-2 \tilde s_2, \quad q=m-\frac{n}2+s_{-1}-s_1-2s_2.$$ 

Let us consider the following two cases:

\noindent Case 1. $\tilde q\geq -1$. For this case, 
using mainly the fact that $\left(\left|b_{k+3}/b_k\right|\right)_{k\ge0}$ is a strictly-increasing sequence 
and dividing the considerations into three subclasses:
$k_n\equiv -m, 1-m, 2-m \, ({\rm mod}\,3)$, one can find that  
\begin{align}\label{bk1bkn}
|b_{k_1}\cdots b_{k_n}|
&\leq 8\max\{|b_{\tilde q}|,|b_{\tilde q+1}|,|b_{\tilde q+2}|\}\left|b_{3g+\frac{3n}2-3-m} \right|,
\end{align}
where the facts that $|b_{k}|\leq2|b_{k+1}|$ for all $k\ge-1$ 
and that $|b_{-1}|, |b_{0}|, |b_{1}|, |b_{2}|$ are all less than or equal to~1
could be used. (Here, we provide for example for the subclass $k_n \equiv -m\,({\rm mod}\,3)$ a few more details: for $i=1,\dots,n-2$ and $k_i\ge0$, it is helpful to write $b_{k_i}=\frac{b_{k_i}}{b_{k_{i}-3}}\frac{b_{k_{i}-3}}{b_{k_{i}-6}}\cdots b_{k_i\, {\rm mod}\, 3}$,
$$b_{k_{n-1}}=\left\{\begin{array}{ll}
\frac{b_{k_{n-1}}}{b_{k_{n-1}+3}}\frac{b_{k_{n-1}+3}}{b_{k_{n-1}+6}}\cdots b_{\tilde{q}},  \quad k_{n-1}\leq \tilde{q},\\
\frac{b_{k_{n-1}}}{b_{k_{n-1}-3}}\frac{b_{k_{n-1}-3}}{b_{k_{n-1}-6}}\cdots b_{\tilde{q}},  \quad k_{n-1}> \tilde{q},
\end{array}
\right.$$
and $b_{k_n}=\frac{b_{k_{n}}}{b_{k_n+3}}\frac{b_{k_n+3}}{b_{k_n+6}}\cdots b_{3g+\frac{3n}2-3-m}$.)
Since $a_{k_1,\dots,k_n}\neq 0$, we know from~\eqref{a} that
 $\tilde s_{-1}\leq n/2$. So $\tilde q\leq m$.
We conclude from \eqref{bk1bkn}  that there exits a 
  constant~$C_2=C_2(m)$, {\it that is independent of $n,k_1,\dots,k_n$}, such that, for sufficiently large~$g$, 
 \begin{equation}\label{case1bb}
 	\left|b_{k_1}\cdots b_{k_n}\right| \leq C_2 \left |b_{3g+\frac{3n}2-3-m}\right|.
 \end{equation}
 
\noindent Case 2. $\tilde q<-1$. This implies $q\leq-1$. Using the fact that $\left(\left|b_{k+3}/b_k\right|\right)_{k\ge0}$ is an increasing sequence, we have 
\beq
\left|b_{k_1}\cdots b_{k_n}\right|\leq 
\left|b_{-1}^{s_{-1}} b_{0}^{s_0} b_1^{s_1} b_2^{s_2} b_{3g+\frac{3n}2-m-3+q} \right|.
\eeq
By using the facts that $|b_{k}|\leq2|b_{k+1}|$, $|b_{-1}|, |b_0|, |b_1|, |b_2|\le1$ and 
that for sufficiently large $k$, $|b_{k+2}|>|b_k|$, 
we further conclude the existence of an absolute constant $C_3$, such that for sufficiently large~$g$,
\beq\label{case2bb}
\left|b_{k_1}\cdots b_{k_n}\right| \leq C_3 \left|b_{3g+\frac{3n}2-3-m}\right|.
\eeq

Next, from the definition~\eqref{defbk}, one can verify that there 
exits a constant $C_4=C_4(m)$, that is {\it independent of~$n$}, such that, for sufficiently large~$g$, 
\begin{align} \label{36}
g^{\left[\frac23 m\right]} \left|b_{3g+\frac{3n}2-3-m}\right| \leq C_4 \left|b_{3g+\frac{3n}2-3}\right|.
\end{align}

Using \eqref{case1bb}, \eqref{case2bb}, \eqref{36} we thus obtain the existence of 
a constant $C_1=C_1(m)$, that is {\it independent of $n,k_1,\dots,k_n$}, 
such that for $g$ sufficiently large,  
\begin{equation}
g^{\left[\frac23 m\right]} \left|b_{k_1}\cdots b_{k_n}\right| \leq C_1\, \frac{(6g+3n-7)!!}{24^{g+\frac{n}2-1}\left(g+\frac{n}2-1\right)!}.
\end{equation}
Lemma~\ref{lem1even0906} is proved.
\end{proof}

The proof of Lemma~\ref{lem1odd0906} is similar to that of Lemma~\ref{lem1even0906}; details are omitted.

\begin{proof}[Proof of Lemma~\ref{lem2}]
Let us assume that 
$a_{k_1,\dots,k_n}\neq 0$ (otherwise trivial). The inequality~\eqref{abinequality} 
can obviously be written as  
\beq\label{acineq}
\left|a_{k_1,\dots,k_n}\right| \, \prod_{j=1}^n(k_j+2)^2\leq 2\left|c_{k_1}\cdots c_{k_n}\right|.
\eeq 
Here, 
\beq
c_k:=b_k \, (k+2)^2, \quad k\ge -1.
\eeq
For example,
$\left|c_{-1}\right|=1$, $\left|c_{0}\right|=4$, $\left|c_{1}\right|=9/2$, $\left|c_2\right|=14$,
$\left|c_{3}\right|=125/8$.
One can verify that $(\left|c_k\right|)_{k\geq0}$ are a strictly-increasing sequence and that 
$\forall\,\ell\geq2$ and $k\geq0$, $\left|c_{k+\ell+3}/c_{k+\ell}\right|\geq\left|c_{k+3}/c_k\right|$.
Thus, 
\begin{equation}\label{ck1ckn}
\left|c_{k_1}\cdots c_{k_n}\right|\leq \left|c_{-1}^{s_{-1}} c_0^{s_{00}} c_1^{s_1} c_2^{s_2} c_3^{s_{03}} c_{3g+n-3+s_{-1}-s_1-2s_2-3 s_{03}}\right|,
\end{equation}
where 
\begin{align}
s_{00} &:={\rm card} \{i\in\{1,\dots,n-1\}\mid k_i=0\},\\
s_{03} &:={\rm card} \{i\in\{1,\dots,n-1\}\mid k_i\geq 3, k_i\equiv0\,({\rm mod}\,3)\}.
\end{align}
Note that $s_0=s_{00}+s_{03}$. From the condition $m\leq {\rm card} \{i\mid k_i\geq 1\}$, we know that $s_1+s_2+s_{03}\geq m-1$.
Using the fact that for sufficiently large $k$, $\left|c_{k+2}/c_k\right|>c_3^2/c_2^2$ and using~\eqref{ck1ckn}, we find that for sufficiently large~$g$,
\begin{align}
\left|c_{k_1}\cdots c_{k_n}\right|
&\leq\left\{\begin{array}{ll}
\left|c_{-1}^{s_{-1}}c_0^{s_{00}}c_1^{s_1}c_2^{s_2+s_{03}}c_{3g+n-3+s_{-1}-s_1-2s_2-2s_{03}}\right|, & s_{03} {\rm~even},\\
\\
\left|c_{-1}^{s_{-1}}c_0^{s_{00}}c_1^{s_1}c_2^{s_2+s_{03}-1}c_3 c_{3g+n-4+s_{-1}-s_1-2s_2-2s_{03}}\right|, &s_{03} {\rm~odd}.
\end{array}\right.
\end{align}
Since 
$\left|c_3/c_2\right|<1.2$ and $(\left|c_k\right|)_{k\geq0}$ is a strictly-increasing sequence, we find
\begin{equation}\label{ck1ckn2}
    	\left|c_{k_1}\cdots c_{k_n}\right|\leq 1.2\left|c_{-1}^{s_{-1}} c_0^{s_{00}} c_1^{s_1} c_2^{s_2+s_{03}}c_{3g+n-3+s_{-1}-s_1-2s_2-2s_{03}}\right|.
\end{equation}

Since $a_{k_1,\dots,k_n}\neq0$ and since~\eqref{a}, we have $s_{-1}+s_2\leq \left[(n-s_1)/2\right]$. So  
$$3g+n-3+s_{-1}-s_1-2s_2-2s_{03}\leq 3g+\left[\frac{3n}2-\frac{3}2(m-1)\right]-3,$$
where we also used $s_{1}+s_{2}+s_{03}\ge m-1$.
Noticing again that $(\left|c_k\right|)_{k\geq0}$ is strictly-increasing 
and using $s_{00}+s_1+s_2+s_{03}+s_{-1} = n-1$,
we have
\begin{align}\label{thm4formula1}
\left|c_{k_1}\cdots c_{k_n}\right|
&\leq 1.2\times14^{n-1}\left|c_{3g-3+\left[\frac{3n}2-\frac{3}2 (m-1)\right]}\right|\\
&\leq 1.2\times14^{n-1}\left|b_{3g-3+\left[\frac{3n}2-\frac{3}2 (m-1)\right]}\right| \left(3g+\left[\frac{3n}{2}\right]\right)^2. \nn
\end{align}
Using \eqref{defbk}, we know that there exists a constant $C_7=C_7(C)$, that is {\it independent of $n,m$}, such that for sufficiently large $g$ 
and $m\leq n\leq C\log(g)$,
\begin{equation}\label{thm4formula3}
g^{m-3}\,\left|b_{3g-3+\left[\frac{3n}2-\frac{3}2 (m-1)\right]}\right| \left(3g+\left[\tfrac{3n}2\right]\right)^2
\leq C_7\, \tfrac{\left(6g+2\left[\frac{3n}2\right]-7\right)!!}{24^{g+\frac{n}2-1}\left(g+\frac{n}2-1\right)!}.
\end{equation}
Combined with~\eqref{thm4formula1}, the lemma is proved.
\end{proof}

\begin{proof}[Proof of Lemma~\ref{lem3}]
Recalling that $\forall\,\ell\geq2$ and $k\geq0$, $\left|c_{k+\ell+3}/c_{k+\ell}\right|\geq\left|c_{k+3}/c_k\right|$, 
the proof will be similar to that of Lemma~\ref{lem1even0906}.
Let $g>0$ be a sufficiently large integer and $2\leq n\leq C\log(g)$ be an even integer. Assume that $a_{k_1,\dots,k_n}\neq 0$ (otherwise trivial). 
Due to~\eqref{acineq} we also assume that $k_n\geq k_{n-1} \geq \max\{k_1,\dots,k_{n-2}\}$.

Consider the following two cases:
	
\noindent Case~1.  $\tilde q\ge-1$. Since $a_{k_1,\dots k_n}\neq0$, using \eqref{a}, we have $\tilde s_{-1}\leq n/2$, thus, 
$\tilde s_{1}+2\tilde s_2\leq m+1$. Similarly to the proof in Lemma~\ref{lem1even0906}, we have
\begin{align}
\left|c_{k_1} \cdots c_{k_n}\right|
&\leq 2^n14^{m+1}\max\{|c_{\tilde q}|,|c_{\tilde q+1}|,|c_{\tilde q+2}|\} \left|c_{3g+\frac{3n}2-3-m}\right|.
\label{ck1cknlem5}
\end{align}
Since $\tilde s_{-1}\leq\frac n2$, we have $\tilde q\leq m$. Combined with~\eqref{ck1cknlem5} we know that
 there exists a constant $C_9=C_9(m)$, that is {\it independent of $n,k_1,\dots,k_n$}, such that
\begin{align}\label{c91012}
|c_{k_1}\cdots c_{k_n}|\leq C_9\,2^n\left|c_{3g+\frac{3n}2-3-m}\right|.
\end{align}

\noindent Case~2. $\tilde q<-1$. This implies $q\leq -1$.
Similarly to the proof in Lemma~\ref{lem1even0906}, we have that 
for sufficiently large $g$ and $n\leq C\log(g)$, 
there exists a constant $C_{10}=C_{10}(m)$, that is {\it independent of $n,k_1,\dots,k_n$}, such that
\begin{align}\label{c101012}
\left|c_{k_1} \cdots c_{k_n}\right|&\leq 4^{s_0} \left(\frac92\right)^{s_1} 14^{s_2} \left|c_{3g+\frac{3n}2-3-m+q}\right| 
\leq C_{10}\,2^n\left|c_{3g+\frac{3n}2-3-m}\right|.
\end{align}

Next, recall from~\eqref{36} that 
\begin{equation}\label{1401012}
	g^{\left[\frac23 m\right]}\left|c_{3g+\frac{3n}2-3-m}\right|\leq C_4\left|c_{3g+\frac{3n}2-3}\right|.
\end{equation}
From \eqref{c91012}, \eqref{c101012}, \eqref{1401012} we conclude that 
\begin{equation}
g^{\left[\frac23 m\right]-2}\left|c_{k_1}\cdots c_{k_n}\right|
\leq C_8\,2^n\,\frac{(6g+3n-7)!!}{24^{g+\frac{n}2-1} \bigl(g+\frac{n}2-1\bigr)!}\frac{\left(3g+\frac{3n}2\right)^2}{g^2},
\end{equation}
where $C_8:=C_4(m) \, \max\{C_9(m),C_{10}(m)\}$. The lemma is proved. 
\end{proof}

The proof of Lemma~\ref{lem3'} is similar to that of Lemma~\ref{lem3}, so we omit its details. 
Before proving Lemma~\ref{lem7} we mention that 
the following estimate is valid:
\begin{align}\label{stronglemma7}
	\lim_{g\to\infty}\max_{{2\leq n\leq C\log(g)}}\max_{d_1,\dots,d_n\ge 0 \atop d_1+\cdots+d_n=3g+n-3}\max_{\sigma\in S_{n,1}^{\{2\}}}2^{n-2}\left|\gamma_{\underline{d},\sigma}(g)-\frac{1}{2^{n-2}}\right|=0,
\end{align}
whose proof is exactly similar to 
the estimates in the proof of Lemma~\ref{lem6}. This statement~\eqref{stronglemma7} is stronger than Lemma~\ref{lem6}. 
Let us now proceed to prove Lemma~\ref{lem7}.

\begin{proof}[Proof of Lemma~\ref{lem7}]
Let us first prove~\eqref{lem7formula1} with $l=2$. 
Consider the case when $n$ is even.
Take $D_{2}=8+[6C]$. 
For every $\sigma\in S_{n,1}^{\{4\}}$, we have
\begin{align}\label{lem8formula}
&  \sum_{\underline{j}\in \left(\ZZ^{\ge0}\right)^n} 
a_{K_{\underline{d},\sigma,1}(\underline{j}),\dots,K_{\underline{d},\sigma,n}(\underline{j})}\\
& =  \left(\sum_{\underline{j} \in V^{3g+\frac{3n}2-D_2}_{\underline{d},\sigma,-1}}  
+ \sum_{\underline{j} \in V^{3g+\frac{3n}2-5}_{\underline{d},\sigma,3g+\frac{3n}2-D_2+1}} \right)  
a_{K_{\underline{d},\sigma,1}(\underline{j}),\dots,K_{\underline{d},\sigma,n}(\underline{j})}. \nn
\end{align}
For $g$ sufficiently large and $2\leq n\leq C\log (g)$,
similarly to the proof of~\eqref{lemma8formula1}, we have 
\begin{align}
& \left|\sum_{\underline{j} \in V^{3g+\frac{3n}2-D_2}_{\underline{d},\sigma,-1}}a_{K_{\underline{d},\sigma,1}(\underline{j}),\dots,K_{\underline{d},\sigma,n}(\underline{j})}\right| \label{formula141}\\
&\leq \frac{C_8(C,D_2-3) \,(6g+3n-7)!!}{24^{g+\frac{n}2-1}\left(g+\frac{n}2-1\right)!} \, \frac{2^n\left(\frac{\pi^2}6\right)^n}{g^{\left[\frac23 (D_2-3)\right]-2}}.\nn
\end{align}
Similarly to~\eqref{lem6estimate5}, we have
\begin{align}
& \left|\sum_{\sigma\in S_{n,1}^{\{4\}}}\sum_{\underline{j} \in V^{3g+\frac{3n}2-5}_{\underline{d},\sigma,3g+\frac{3n}2-D_2+1}} 
a_{K_{\underline{d},\sigma,1(\underline{j})},\dots,K_{\underline{d},\sigma,1(\underline{j})}} \right| \label{formula142}\\
& \leq \frac{2^n(6g+3n-7)!!}{24^{g+\frac{n}2-1}\left(g+\frac{n}2-1\right)!} 
\sum_{f=5}^{D_2-1} \bigl|\widetilde{A}(f,n)\bigr| \frac{C_1(f-3)}{g^{\left[\frac23(f-3)\right]}},\nn
\end{align}
where for every $f=5,\dots,D_2-1$, $\widetilde{A}(f,n)$ is a certain polynomial of $n$.
Let us give more details for the proof of~\eqref{formula142}. Firstly, observe that for each $f\ge5$, and for $k_1,\dots,k_n\ge-1$ satisfying $k_1+\cdots+k_n=3g+n-3$, $a_{k_1,\dots,k_n}\neq0$ and $\max_{1\leq i\leq n}k_i=3g+\frac {3n}2-f$, the equations
$$K_{\underline{d},\sigma,q}(\underline{j})=k_q, \quad1\leq q\leq n$$
for $\underline{j}\in\big(\ZZ^{\ge0}\big)^n$ could have solutions only when the distance of the two peaks of~$\sigma$ is bounded by~$4(f-3)$. The number of $\sigma\in S_{n,1}^{\{4\}}$ with this bounded distance of 
its two peaks is controlled by $2^n\tilde{B}(f,n)$, where $\tilde{B}(f,n)$ is a certain polynomial of $n$. 
Secondly, we observe that the number of $k_1,\dots,k_n\ge-1$ satisfying $k_1+\cdots+k_n=3g+n-3$, $a_{k_1,\dots,k_n}\neq0$ and $\max_{1\leq i\leq n}k_i=3g+\frac {3n}2-f$ is controlled by a certain polynomial of~$n$. 
Thirdly, for each such $(k_1,\dots,k_n)$, the equations
$K_{\underline{d},\sigma,q}(\underline{j})=k_q$, $1\leq q\leq n$ for $\underline{j}\in \left(\ZZ^{\ge0}\right)^n$ have less than $f-2$ solutions.

The above estimates can be done in a similar way when $n$ is odd. 
From these estimates one gets the validity of~\eqref{lem7formula1} with $l=2$.

Let us proceed to prove~\eqref{lem7formula1} with $l\ge3$.
Consider $n$ is even (again the case with $n$ odd is similar). Take $D_l=[3\bigl(5+C\log(4\pi^2l/3)\bigr)/2]$. 
(It satisfies that $C\log (4\pi^2l/3)<\left[2(D_l-3)/3\right]-2$.)
We have 
\begin{align}\label{lem8formula'}
&  \sum_{\underline{j}\in \left(\ZZ^{\ge0}\right)^n} 
a_{K_{\underline{d},\sigma,1}(\underline{j}),\dots,K_{\underline{d},\sigma,n}(\underline{j})}\\
& =  \left(\sum_{\underline{j} \in V^{3g+\frac{3n}2-D_l}_{\underline{d},\sigma,-1}}  
+ \sum_{\underline{j} \in V^{3g+\frac{3n}2-6}_{\underline{d},\sigma,3g+\frac{3n}2-D_l+1}}\right)  
a_{K_{\underline{d},\sigma,1}(\underline{j}),\dots,K_{\underline{d},\sigma,n}(\underline{j})}. \nn
\end{align}
The rest of the proof of~\eqref{lem7formula1} is similar to the above for $l=2$.

Let us now prove~\eqref{lem7formula2}. For $\sigma\in S_{n,1}^{\{2l\}}$ ($l\geq3$),  
using \eqref{defSn2l1}--\eqref{defSn2l2} we find
$$K_{\underline{d},\sigma,t_{2i}(\underline{j})}=d_{\sigma(t_{2i})}+j_{t_{2i}}+j_{t_{2i-1}}+1\geq1,\quad \forall \, i=1,\dots,l.$$
So by applying Lemma~\ref{lem2} we have
\begin{align}
&\Biggl|\sum_{\underline{j}\in V_{\underline{d},\sigma,-1}^{3g+\left[\frac{3n}2\right]-3}} a_{K_{\underline{d},\sigma,1}(\underline{j}),\dots,K_{\underline{d},\sigma,n}(\underline{j})}\Biggr|\label{lem7proof}\\
&\leq
\tfrac{\left(6g+2\left[\frac{3n}2\right]-7\right)!!}{24^{g+\left[\frac{n}2\right]-1} \left(g+\left[\frac{n}2\right]-1\right)!}
\sum_{\underline{j}\in V_{\underline{d},\sigma,-1}^{3g+\left[\frac{3n}2\right]-3}} \left|\kappa_{K_{\underline{d},\sigma,1}(\underline{j}),\dots,K_{\underline{d},\sigma,n}(\underline{j})}\right|\nn\\
&\leq \tfrac{\left(6g+2\left[\frac{3n}2\right]-7\right)!!}{24^{g+\left[\frac{n}2\right]-1}\left(g+\left[\frac{n}2\right]-1\right)!} 
\sum_{\underline{j} \in V_{\underline{d},\sigma,-1}^{3g+\left[\frac{3n}2\right]-3}} 
\tfrac{C_6(C)\,14^n}{\prod_{q=1}^n(K_{\underline{d},\sigma,q}(\underline{j})+2)^2} \, g^{3-l}\nn\\
&\leq \tfrac{\left(6g+2\left[\frac{3n}2\right]-7\right)!!}{24^{g+\left[\frac{n}2\right]-1} \left(g+\left[\frac{n}2\right]-1\right)!} 
\sum_{j_{t_2}=0}^{3g+\left[\frac{3n}2\right]-3}\sum_{k_1,\dots,\widehat{k_{t_2}},\dots,k_{n}=-1}^{3g+\left[\frac{3n}2\right]-3} 
\tfrac{C_6(C) \,14^n(k_{t_2}+2)^2}{(j_{t_2}+3)^2\prod_{q=1}^n(k_q+2)^2} \, g^{3-l}\nn\\
&\leq \tfrac{\left(6g+2\left[\frac{3n}2\right]-7\right)!!}{24^{g+\left[\frac{n}2\right]-1} \left(g+\left[\frac{n}2\right]-1\right)!} 
\tfrac{C_6(C) \,14^n \, \bigl(\frac{\pi^2}6\bigr)^n}{g^{l-3}}.\nn
\end{align}
Now by using~\eqref{lem7proof} and 
the fact that there exists a constant $C_{13}=C_{13}(C)$, {\it that is independent of $n,g$}, such that for sufficiently large $g$,
\begin{equation}
\frac{2^{n-2} \left(6g+2\left[\frac{3n}2\right]-7\right)!!}{24^{g+\left[\frac{n}2\right]-1} 
	\left(g+\left[\frac{n}2\right]-1\right)!}\leq C_{13}\frac{(6g+2n-5)!!}{24^g g!}
\end{equation}
holds for all $2\leq n\leq C\log(g)$, we get the validity of~\eqref{lem7formula2}.
\end{proof}
\end{appendix}


\begin{thebibliography}{99}

\bibitem{Agg1}
A.~Aggarwal, 
Large genus asymptotics for intersection numbers and principal strata volumes of quadratic differentials.
Invent. Math. {\bf 226} (2021), 897--1010.

\bibitem{AIS}
A.~Alexandrov, F.H.~Iglesias, S. Shadrin, Buryak-Okounkov formula for the $n$-point function 
and a new proof of the Witten conjecture. IMRN~{\bf 2021}, 14296--14315.

\bibitem{AvM}
M. Adler, P. van Moerbeke, A matrix integral solution to two-dimensional $W_p$-gravity. 
Comm. Math. Phys.~{\bf 147} (1992), 25--56.

\bibitem{BDY1} 
M. Bertola, B. Dubrovin, and D. Yang, 
Correlation functions of the KdV hierarchy and applications to 
intersection numbers over $\overline{\mathcal{M}}_{g,n}$. Physica D. Nonlinear Phenomena {\bf 327} (2016), 30--57.

\bibitem{BDY2}
M. Bertola, B. Dubrovin, D. Yang, 
Simple Lie algebras and topological ODEs. IMRN {\bf 2018}, 1368--1410.

\bibitem{BDY3}
M. Bertola, B. Dubrovin, D. Yang, 
Simple Lie algebras, Drinfeld-Sokolov hierarchies, and multi-point correlation functions. 
Mosc. Math. J. {\bf 21} (2021), 233--270.

\bibitem{Bur}
A.~Buryak, Double ramification cycles and the $n$-point function for the moduli space of curves. 
Mosc. Math. J.~{\bf 17} (2017), 1--13.

\bibitem{CMZ}
D. Chen, M. M\"oller, D. Zagier, Quasimodularity and large genus limits of Siegel-Veech constants. J. Amer. Math. Soc. 
{\bf 31} (2018), 1059--1163.

\bibitem{DGZZ20}
V. Delecroix, \'E. Goujard, P. Zograf, A. Zorich, Uniform lower bound for intersection numbers of psi-classes. 
Symmetry Integrability Geom. Methods Appl. {\bf 16} (2020), Paper No.~086, 13~pp.

\bibitem{DGZZ19}
V. Delecroix, \'E. Goujard, P. Zograf, A. Zorich, Masur--Veech Volumes, 
Frequencies of Simple Closed Geodesics, and Intersection Numbers on Moduli Spaces of Curves. Duke Math J.~{\bf 170} (2021), 2633--2718.

\bibitem{DGZZ22}
V. Delecroix, \'E. Goujard, P. Zograf, A. Zorich,
Large genus asymptotic geometry of random square-tiled surfaces and of random multicurves. 
Invent. Math. (2022),
https://doi.org/10.1007/s00222-022-01123-y.

\bibitem{DM69}
P. Deligne, D. Mumford, The irreducibility of the space of curves of given genus. Inst. Hautes \'Etudes Sci. Publ. Math. No.~{\bf 36} (1969), 75--109.

\bibitem{DVV}
R. Dijkgraaf, E. Verlinde, H. Verlinde, Loop equations and Virasoro constraints
in non-perturbative 2-D quantum gravity. Nucl. Phys. B~{\bf 348} (1991), 435--456.

\bibitem{DVY}
B. Dubrovin, D. Valeri, D. Yang, 
Affine Kac--Moody algebras and tau-functions for the Drinfeld--Sokolov hierarchies: 
The matrix-resolvent method. arXiv:2110.06655.

\bibitem{DY1}
B. Dubrovin, D. Yang, Generating series for GUE correlators. Lett. Math. Phys. {\bf 107} (2017), 1971--2012.

\bibitem{DY2}
B. Dubrovin, D. Yang, 
On Gromov-Witten invariants of $\mathbb{P}^1$. Math. Res. Lett. {\bf 26} (2019), 729--748.

\bibitem{DY3}
B. Dubrovin, D. Yang, Matrix resolvent and the discrete KdV hierarchy. 
Comm. Math. Phys. {\bf 377} (2020), 1823--1852.

\bibitem{DYZ0}
B. Dubrovin, D. Yang, D. Zagier, Gromov-Witten invariants of the Riemann sphere. 
Pure Appl. Math. Q. {\bf 16} (2020), 153--190.

\bibitem{DYZ}
B. Dubrovin, D. Yang, D. Zagier, 
On tau-functions for the KdV hierarchy. Selecta Math. {\bf 27} (2021), Paper No. 12, 47~pp.

\bibitem{DYZ2}
B. Dubrovin, D. Yang, D. Zagier, Geometry and arithmetic of integrable hierarchies of KdV type. 
I. Integrality. arXiv:2101.10924.

\bibitem{Guo}
J.~Guo, A remark on equivalence between two formulas of the two point Witten-Kontsevich correlators. arXiv:2102.10761.

\bibitem{KL}
M. Kazarian, S. Lando, An algebro-geometric proof of Witten's conjecture. 
J. Amer. Math. Soc. {\bf 20} (2007), 1079--1089.

\bibitem{K}
M. Kontsevich, Intersection theory on the moduli space of curves and the matrix Airy function. 
Comm. Math. Phys. {\bf 147} (1992), 1--23.

\bibitem{LX0}
K. Liu, H. Xu, 
Mirzakhani's recursion formula is equivalent to the Witten-Kontsevich theorem. 
Ast\'erisque {\bf 328} (2009), 223--235.

\bibitem{LX} 
K. Liu, H. Xu, A remark on Mirzakhani's asymptotic formulae. Asian J. Math.~{\bf 18} (2014), 29--52.

\bibitem{M}
M. Mirzakhani, Weil--Petersson volumes and intersection theory on the moduli space of curves. 
J. Amer. Math. Soc. {\bf 20} (2007), 1--23.

\bibitem{O}
A.~Okounkov, Generating functions for intersection numbers on moduli spaces of curves. 
IMRN~{\bf 2002}, 933--957.

\bibitem{OP}
A. Okounkov, R. Pandharipande, Gromov--Witten theory, Hurwitz numbers, and matrix models. In: Proc. Symposia Pure Math., Vol.~{\bf 80}, Part 1, pp. 325--414, 2009.

\bibitem{Wi90}
E. Witten, Two-Dimensional Gravity and Intersection Theory on Moduli Space. 
Surveys in Differential Geometry (1991), pp. 243--320. Lehigh Univ, Bethlehem.

\bibitem{YZZ}
D. Yang, D. Zagier, Y. Zhang, Masur-Veech volumes of quadratic differentials and their asymptotics. 
J. Geom. Phys.~{\bf 158} (2020), 103870, 12~pp.

\bibitem{Zhou}
J. Zhou, Emergent geometry and mirror symmetry of a point. arXiv:1507.01679.

\bibitem{Zograf}
P.G.~Zograf, An explicit formula for Witten's 2-correlators. J. Math. Sci. {\bf 240} (2019), 535--538.

\end{thebibliography}
\end{document}